\documentclass{birkau}

\usepackage[mathscr]{eucal}
\usepackage{enumerate,graphicx,color}
\usepackage{url}
\usepackage{amssymb}
\usepackage{amsmath}
\usepackage{amsthm}
\usepackage{xy}
\usepackage{comment} 

\numberwithin{equation}{section}

\theoremstyle{plain}
\newtheorem{thm}{Theorem}[section]
\newtheorem{lem}[thm]{Lemma}
\newtheorem{prop}[thm]{Proposition}
\newtheorem{coro}[thm]{Corollary}

\theoremstyle{definition}
\newtheorem{df}[thm]{Definition}
\newtheorem{rem}[thm]{Remark}
\newtheorem{eg}[thm]{Example}

\newtheorem{prob}{Problem}

\usepackage{pgf,tikz}

\usetikzlibrary{calc,positioning,shapes,arrows,arrows.meta}

\tikzset{%
world/.style={circle,draw,minimum size=0.5cm,fill=gray!15},
label/.style={shape=rectangle, inner sep=6pt},
shaded/.style={draw, shape=circle, fill=black!35, inner sep=1.4pt},
unshaded/.style={draw, shape=circle, fill=white, inner sep=1.4pt},
quasi/.style={draw, shape=rectangle, rounded corners=3pt, fill=white, inner sep=2.5pt, minimum height=14.5pt},
blob/.style={draw, shape=rectangle, rounded corners=12pt, thin, densely dotted},
order/.style={thin},
curvy/.style={thin, looseness=1.2, bend angle=70},
fatcurvy/.style={thin, looseness=1.7, bend angle=75},
map/.style={->, densely dashed, shorten >=5pt, shorten <=5pt, >=stealth', looseness=1.1},
operationgj/.style={->, densely dashed, shorten >=5pt, shorten <=18pt, >=stealth', looseness=1.1},
relationlejk/.style={->, shorten >=5pt, shorten <=5pt, >=stealth'},
auto}

\newcommand{\eusbA}{\medsub e {\kern-0.75pt\A\kern-0.75pt}}

\newcommand{\bbar}[1]{{\underline{\mathbf{#1}}}}
\newcommand{\twiddle}[1]{{\smash{\underset{\raise.375ex\hbox{$\smash\sim$}}
       {\mathbf{#1}}}\vphantom{\underline{\mathbf{#1}}}}} 
\newcommand{\stwiddle}[1]{\smash{\underset{\smash{\raise.1ex\hbox{\small$\sim$}}}
                         {\mathbf{#1}}}\vphantom{#1}}
\newcommand{\twoB}{\bbar{2}}
\newcommand{\twoT}{\twiddle 2}

\newcommand{\mpe}[2]{\CG^{\rm mp}(#1,#2)}
\newcommand{\mph}[2]{\CL^{\rm mp}(#1,#2)}
\newcommand{\mpm}[2]{\CG_\T^{\rm mp}(#1,#2)}

\newcommand{\cat}[1]{\boldsymbol{\mathscr{#1}}}

\newcommand{\CL}{\cat L}
\newcommand{\CG}{\cat G}

\newcommand{\Lalg}{\mathbf{L}}

\font\bmi=cmmi8 scaled 1440
\newcommand{\powerset}{\raise.6ex\hbox{\bmi\char'175 }}

\newcommand{\T}{\mathscr{T}}

\DeclareMathOperator{\dom}{dom} 

\renewcommand{\le}{\leqslant}

\newcommand{\MDFIP}[2]{\langle {\uparrow} #1,{\downarrow} #2\rangle}
\newcommand{\mdfip}[2]{\langle  #1,  #2\rangle}

\newcommand{\varex}{\varepsilon_x}

\begin{document}
\title{Dual spaces of lattices and semidistributive lattices}

\author{Andrew Craig}
\address{Department of Mathematics and Applied Mathematics\\
University of Johannesburg\\PO Box 524, Auckland Park, 2006\\South~Africa}
\urladdr{https://orcid.org/0000-0002-4787-3760}

\email{acraig@uj.ac.za}

\author{Miroslav Haviar}
\address{Department of Mathematics\\
Faculty of Natural Sciences\\
M. Bel University\\
Tajovsk\'{e}ho 40, 974 01 Bansk\'{a} Bystrica\\
Slovakia\\and\\Department of Mathematics and Applied Mathematics\\
University of Johannesburg\\PO Box 524, Auckland Park, 2006\\South~Africa}
\urladdr{https://orcid.org/0000-0002-9721-152X}
\email{miroslav.haviar@umb.sk}

\author{Jos\'{e} S\~{a}o Jo\~{a}o}
\address{Institut de Recherche Math\'{e}matique Avanc\'{e}e\\ 
UMR 7501, Université de Strasbourg\\
7 rue René-Descartes\\
67084 Strasbourg CEDEX\\France}
\urladdr{https://orcid.org/0000-0001-5078-362X}
\email{saojoao@unistra.fr}

\thanks{The first author acknowledges the National Research Foundation (NRF) of South Africa grant
127266 as well as the Slovak National Scholarship Programme.
The second author acknowledges support by Slovak VEGA grant 1/0152/22.}

\subjclass{06B15, 06A75, 05C20}

\keywords{join semidistributive lattice, meet semidistributive lattice, semidistributive lattice, 
dual digraph, TiRS digraph}

\begin{abstract}
Birkhoff's 1937 dual representation of finite distributive lattices via finite posets was in 1970  extended to a dual representation of arbitrary distributive lattices via compact totally order-disconnected topological spaces by Priestley. This result enabled the development of natural duality theory in the 1980s by Davey and Werner, later on also in collaboration with Clark and Priestley. In 1978 Urquhart extended Priestley's representation to general lattices via compact doubly quasi-ordered topological spaces (L-spaces). In 1995 Plo\v{s}\v{c}ica presented Urquhart's representation in the spirit of natural duality theory by replacing, on the dual side, Urquhart's two quasiorders with a digraph relation generalising Priestley's order relation. In this paper we translate, following the spirit of natural duality theory, Urquhart's L-spaces into newly introduced \emph{Plo\v{s}\v{c}ica spaces}. We then prove that every Plo\v{s}\v{c}ica space is the dual space of some general lattice. Based on the authors' 2022 characterisation of finite join and meet semidistributive lattices via their dual digraphs, we initiate a study of general (possibly infinite) join and meet semidistributive lattices via their dual digraphs. We illustrate our results on examples and formulate three open problems.
\end{abstract}

\dedicatory{Dedicated to Hilary Priestley and her work in duality theory.}

\maketitle


\section{Introduction}\label{sec:intro}

In 1937 Birkhoff~\cite{B1937} showed that every finite distributive lattice $\Lalg$ can be represented as the lattice of all downsets of the 
poset  $(J(\Lalg), \leqslant)$  of join-irreducible elements of $\Lalg$, with the ordering inherited from $\Lalg$. It is easy to show that also every finite poset $(P, \le)$ can be represented as 
the poset of join-irreducible elements of the distributive lattice of all downsets of $(P, \le)$. The Birkhoff dual representation of finite distributive lattices via finite posets was extended to a dual representation of arbitrary distributive lattices (with 
least element $0$ and greatest element $1$) by Priestley in 1970~\cite{Pr70}.

For general lattices (with bounds $0$ and $1$) their first well-known dual representation was presented by Urquhart in 1978~\cite{U78}. After that, a number of authors have attempted to provide various different dual representations of general lattices; for a summary of these representations we refer to a survey paper by the first author~\cite{C22} and also the recent preprint~\cite{BCM25}. In the present paper we will rely on a representation of general lattices (with bounds) by Plo\v{s}\v{c}ica, who in 1995~\cite{Plos95} presented Urquhart's representation of lattices in the spirit of natural duality theory in the sense of Davey and Werner~\cite{DW83}, and Clark and Davey~\cite{CD98}.

In Urquhart's representation of a general lattice $\Lalg$ the elements of the dual space are maximal disjoint filter-ideal pairs (briefly \emph{MDFIP}s) of the lattice $\Lalg$. Urquhart considered two quasi-orders $\le_1$ and $\le_2$ on the set $X_{\Lalg}$ of MDFIPs and presented the dual of the lattice $\Lalg$ as a certain doubly quasi-ordered space $\mathcal{S}(\Lalg)=(X_\Lalg,\le_1,\le_2,\T_\Lalg)$ with a compact topology $\T_\Lalg$, called an \emph{L-space}. 
He introduced an abstract L-space $\mathbb{S}=(X,\le_1,\le_2,\T)$ and the concept of \emph{doubly closed stable sets} of $\mathbb{S}$. He proved that every lattice $\Lalg$ (with bounds) is isomorphic to the lattice $\mathcal{L}(\mathcal{S}(\Lalg))$ of doubly closed stable sets of the dual L-space 
$\mathcal{S}(\Lalg)$~\cite[Theorem 1]{U78}. He also showed that conversely, every L-space $\mathbb{S}$ is isomorphic to the dual space of the lattice $\mathcal{L}(\mathbb{S})$~\cite[Theorem 2]{U78}. Urquhart's dual representation of general lattices has not been used much in practice for representing lattices: the reason might be that his dual of a lattice is a somewhat complicated structure of a doubly quasi-ordered space and the concept of doubly closed stable sets is not easy to work with.

In Plo\v{s}\v{c}ica's representation~\cite{Plos95}, the dual space $\mathcal{D}(\Lalg)=(P_{\Lalg},E,\T_\Lalg)$ of a lattice $\Lalg$ is given by the set $P_{\Lalg}$ of maximal partial homomorphisms (briefly \emph{MPH}s) from $\Lalg$ into the two-element lattice $\twoB$, which correspond to Urquhart's MDFIPs of $\Lalg$. In case the lattice $\Lalg$ is distributive, these MPHs become total homomorphisms from $\Lalg$ into $\twoB$ and they form the Priestley dual of $\Lalg$~\cite{Pr70,Pr72}. As we emphasised already in~\cite{P1}, the close relationship between Plo\v{s}\v{c}ica's representation of general lattices and Priestley's representation of distributive lattices lies in the single binary relation $E$, which Plo\v{s}\v{c}ica considered on his dual space: when $\Lalg$ is distributive, the relation $E$ becomes the Priestley order on the dual space.

Plo\v{s}\v{c}ica's dual space of a general lattice $\Lalg$ is therefore a digraph where the vertices are the MPHs from $\Lalg$ into $\twoB$. The binary relation $E$, which mimics Priestley's order, forms the edge set of the digraph. These duals were presented and studied as TiRS digraphs in two papers by Craig, Gouveia and Haviar~\cite{P3,P4}.

Urquhart's representation of a general lattice $\Lalg$ (with bounds) as the lattice $\mathcal{L}(\mathcal{S}(\Lalg))$ was translated by Plo\v{s}\v{c}ica into his setting using his dual space $\mathcal{D}(\Lalg)=(P_{\Lalg},E,\T_\Lalg)$; we refer to~\cite[Theorem 1.7]{Plos95}. He showed that $\Lalg$ is isomorphic to $\mathcal{L}(\mathcal{D}(\Lalg))$, where the symbol $\mathcal{L}$ now denotes the natural evaluation maps on the dual $\mathcal{D}(\Lalg)$ (see Theorem~\ref{thm:rep-for-L} in Section~\ref{sec:ploscicaspaces}). Yet Plo\v{s}\v{c}ica did not present the equivalent description of Urquhart's abstract L-spaces $\mathbb{S}=(X,\le_1,\le_2,\T)$ in his setting.  Moreover, Plo\v{s}\v{c}ica did not translate into his setting the result of Urquhart saying that each L-space $\mathbb{S}=(X,\le_1,\le_2,\T)$ is isomorphic to the dual space of the lattice $\mathcal{L}(\mathbb{S})$. We complete both these unfulfilled tasks of Plo\v{s}\v{c}ica in Section~\ref{sec:ploscicaspaces} of this paper. We firstly provide an equivalent description of Urquhart's abstract L-spaces in Plo\v{s}\v{c}ica's setting and we call these objects \emph{Plo\v{s}\v{c}ica spaces} (see Definition~\ref{def:ploscicaspaces}). These are TiRS digraphs  $\mathbb{P}=(X,E,\T)$ with an edge set $E$ and a compact topology~$\T$. Then we prove that every Plo\v{s}\v{c}ica space $\mathbb{P}=(X,E,\T)$ is isomorphic to the space $\mathcal{D}(\mathcal{L}(\mathbb{P}))= (P_{\mathcal{L}(\mathbb{P})},E, 
\T_{\mathcal{L}(\mathbb{P})})$ dual to the lattice $\mathcal{L}(\mathbb{P})$. Here the lattice $\mathcal{L}(\mathbb{P})$ dual to $\mathbb{P}$ is formed, instead of using Urquhart's doubly closed stable sets, by using the equivalent concept of Plo\v{s}\v{c}ica's maximal partial morphisms (briefly \emph{MPM}s) from $\mathbb{P}$ into the two-element digraph with the discrete topology.

In Section~\ref{sec:dualspacesSDL} of this paper we apply the combined approach of~\cite{P5} by using Urquhart's MDFIPs for the elements of the dual of a general lattice $\Lalg$ and we study such a dual of $\Lalg$ as a TiRS digraph $(X_\Lalg,E)$ using the Plo\v{s}\v{c}ica binary relation $E$ on the vertices. We recall that in~\cite{P5} we characterized the dual digraphs of finite join and meet semidistributive lattices (the topology plays no role in the finite case). Our results relied on a characterization of finite join and meet semidistributive lattices by Adaricheva and Nation~\cite[Theorem 3-1.4]{AN-Ch3}. Yet this characterization cannot be generalized to the infinite case (cf.~\cite[Theorem 3-1.27]{AN-Ch3}).

Section~\ref{sec:dualspacesSDL} aims to identify a possible characterisation of join and meet semidistributive lattices by their dual digraphs. This is a challenging task and it requires a different method to that used in~\cite{P5}.

We are able to show that if the dual digraph $(X_\Lalg,E)$ of a general lattice $\Lalg$ (with bounds) contains no two distinct MDFIPs with the same ideal, then $\Lalg$ is join semidistributive. Hence we give a natural sufficient condition for join semidistributivity of a general lattice $\Lalg$. Dually, 
one can then obtain that a general lattice $\Lalg$ (with bounds) is meet semidistributive provided its dual digraph $(X_\Lalg,E)$ contains no two different MDFIPs with the same filter.

The main advantage of Plo\v{s}\v{c}ica's dual representation is in our opinion the use of the binary relation $E$ on the first duals of lattices, which therefore can be studied as digraphs, and the use of the MPMs as elements of the second duals of lattices instead of the doubly closed stable sets in Urquhart's representation. We see as much less important for any future user of the Plo\v{s}\v{c}ica representation of general lattices (with bounds), as we present it here, whether for the elements of the first dual of a lattice $\Lalg$, the MPHs from $\Lalg$ to $\twoB$ or their corresponding MDFIPs of $\Lalg$ are used. Both approaches might be equally employed although in certain situations one of them can be seen as more natural than the other. In Section~\ref{sec:ploscicaspaces} we prefer  Plo\v{s}\v{c}ica's MPHs from $\Lalg$ to $\twoB$ as the elements of the first dual of $\Lalg$ since they well interact with evaluation maps used often in this section while translating Urquhart's dual representation into Plo\v{s}\v{c}ica's setting. On the other hand, in Section~\ref{sec:dualspacesSDL} we employ the corresponding MDFIPs of $\Lalg$ when studying the results of~\cite{P5} for general infinite
lattices (with bounds) since MDFIPs were already naturally used in~\cite{P5}.

At the end of the paper we present three open problems that could attract the attention of those interested in future research in this topic.

\section{Preliminaries}
\label{sec:prelim}

Here we lay out the necessary preliminary definitions and results that we will need later on.

A \emph{partial homomorphism} from a lattice $\Lalg = (L,\wedge,\vee,0,1)$ into the two-element lattice $\twoB=  
(\{0,1\},\wedge,\vee,0,1)$ is a partial map $f : L \to \{0,1\}$ such that $\dom f$ is a bounded sublattice of $\Lalg$ and the restriction $f\!\! \upharpoonright_{\dom f}$ is a lattice homomorphism preserving the bounds. Then a \emph{maximal partial homomorphism} (MPH) is a partial homomorphism with no proper extension. By $\mph{\Lalg}{\twoB}$ we denote the set of all MPHs from $\Lalg$ into~$\twoB$.

\begin{df}[{\cite[Section 3]{U78}}]\label{def:DFIP}
Let $\Lalg$ be a lattice. Then $\langle F, I \rangle$ is a \emph{disjoint filter-ideal pair} of $\Lalg$ if $F$ is a filter of $\Lalg$ and $I$ is an ideal of $\Lalg$ such that $F \cap I = \varnothing$. We say that a disjoint filter-ideal pair $\langle F, I \rangle$ is maximal if there is no disjoint filter-ideal pair $\langle G, J\rangle \neq \langle F, I \rangle$ such that $F \subseteq G$ and $I \subseteq J$. 
\end{df}

It is well-known that for a lattice $\Lalg = (L,\wedge,\vee,0,1)$ (with bounds) there is a one-to-one correspondence between the set of MPHs from $\Lalg$ to $\twoB$ and the MDFIPs of $\Lalg$. (See e.g. \cite[p.~76]{Plos95}.) Indeed, for an MPH $f$ from $\Lalg$ to $\twoB$, $\langle f^{-1}(1),f^{-1}(0) \rangle$ is an MDFIP of $\Lalg$. Conversely, for any MDFIP $\langle F,I \rangle$ of $\Lalg$, the partial function $f$ from $\Lalg$ to $\twoB$ given by $f^{-1}(1)=F$ and $f^{-1}(0)=I$ is an MPH.

We recall (see~\cite{Plos95}) that Plo\v{s}\v{c}ica's binary relation on the set $P_{\Lalg}$ of MPHs from $\Lalg$ to $\twoB$, which are used in Section~\ref{sec:ploscicaspaces}, is defined as follows: for any  MPHs  $f, g$ from $\Lalg$ to $\twoB$,

\begin{itemize}
\item[(E1)] $fEg \quad
\iff \quad (\forall x \in \dom f \cap \dom g)(f(x) \le g(x))$.
\end{itemize}

The base set of our dual space to $\Lalg$ will in Section~\ref{sec:dualspacesSDL} 
be the set $X_{\Lalg}$ of all MDFIPs of $\Lalg$. For two MDFIPs $\langle F, I \rangle$ and $\langle G, J \rangle$, Plo\v{s}\v{c}ica's relation $E$ is determined on the set $X_{\Lalg}$ as follows (cf.~\cite[p.~373]{P5}): 

\begin{itemize}
\item[(E2)] $\langle F, I \rangle E \langle G, J \rangle \quad
\iff \quad F \cap J = \emptyset$.
\end{itemize}

\begin{figure}[ht]
\centering
\begin{tikzpicture}[scale=0.7]

\begin{scope}[]
\node[unshaded] (bot) at (0,0) {};
\node[unshaded] (a) at (-1.5,1.5) {};
\node[unshaded] (b) at (0,1.5) {};
\node[unshaded] (c) at (1.5,1.5) {};
\node[unshaded] (top) at (0,3) {};
\draw[order] (bot)--(a)--(top)--(b)--(bot)--(c)--(top);
\node[label,anchor=east,xshift=1pt] at (a) {$a$};
\node[label,anchor=east,xshift=1pt] at (b) {$b$};
\node[label,anchor=east,xshift=1pt] at (c) {$c$};
\node[label,anchor=north] at (bot) {$0$};
\node[label,anchor=south] at (top) {$1$};
\node[label,anchor=north,yshift=-10pt, xshift=1.75cm] at (bot){$\mathbf{M}_3$ and $(X_{\mathbf{M}_3}, E)$};
\end{scope}

\begin{scope}[xshift=4.25cm]
\node[] (ca) at (-0.75,0) {$ca$};
\node[] (ba) at (0.75,0) {$ba$};
\node[] (cb) at (-1.5,1.25) {$cb$};
\node[] (bc) at (1.5,1.25) {$bc$};
\node[] (ab) at (-0.75,2.5) {$ab$};
\node[] (ac) at (0.75,2.5) {$ac$};

\path[thick,<->] (ab) edge (ac);
\path[thick,<->] (ac) edge (bc);
\path[thick,<->] (bc) edge (ba);
\path[thick,<->] (ba) edge (ca);
\path[thick,<->] (ca) edge (cb);
\path[thick,<->] (cb) edge (ab);

\path[thick,->] (ab.south east) [out=300,in=155] edge  (bc.west);
\path[thick,->] (bc) [out=195,in=45] edge (ca.north east);
\path[thick,->,shorten >=5pt,shorten <=4pt] (ca.north) edge (ab.south);

\path[thick,->,] (ac.south west) [out=240,in=45] edge  (cb.east);
\path[thick,->] (cb) [out=330,in=135] edge  (ba.north west);
\path[thick,->,shorten >=5pt,shorten <=4pt] (ba.north) edge  (ac.south);

\end{scope}

\begin{scope}[yshift=-6cm, xshift=1cm]
\node[unshaded] (bot) at (0,0) {};
\node[unshaded] (a) at (-1,2) {};
\node[unshaded] (b) at (0,2) {};
\node[unshaded] (c) at (1,2) {};
\node[unshaded] (d) at (-1,1) {};
\node[unshaded] (e) at (1,1) {};
\node[unshaded] (top) at (0,3) {};
\draw[order] (top)--(a)--(d)--(bot);
\draw[order] (b)--(e);
\draw[order] (d)--(b)--(top)--(c)--(e)--(bot);
\node[label,anchor=east,xshift=1pt] at (a) {$a$};
\node[label,anchor=north] at (b) {$b$};
\node[label,anchor=west,xshift=-2pt] at (c) {$c$};
\node[label,anchor=east,xshift=1pt] at (d) {$d$};
\node[label,anchor=west,xshift=-2pt] at (e) {$e$};

\node[label,anchor=north, yshift=-10pt, xshift=1.25cm] at (bot)
{$\mathbf{L}_2$ and $(X_{\mathbf{L}_2}, E)$}; 

\node[label,anchor=north] at (bot) {$0$};
\node[label,anchor=south] at (top) {$1$};
\end{scope}
\begin{scope}[xshift=4.25cm, yshift = -6cm]
\node[] (dc) at (0,0) {$dc$};
\node[] (ab) at (0,1.25) {$ab$};
\node[] (cb) at (0,2.5) {$cb$};
\node[] (ea) at (0,3.75) {$ea$};
\path[thick,<-] (dc.north) edge  (ab.south);
\path[thick,<->] (ab.north) edge  (cb.south);
\path[thick,->] (cb.north) edge (ea.south);
\end{scope}

\begin{scope}[xshift = 9cm]
\node[unshaded] (bot) at (0,0) {};
\node[unshaded] (a) at (-1,2) {};
\node[unshaded] (b) at (1,2) {};
\node[unshaded] (c) at (-1,1) {};
\node[unshaded] (d) at (0,1) {};
\node[unshaded] (e) at (1,1) {};
\node[unshaded] (top) at (0,3) {};
\draw[order] (bot)--(c)--(a)--(d)--(b)--(top)--(a);
\draw[order] (d)--(bot)--(e)--(b);
\node[label,anchor=north, yshift=-10pt, xshift=1.25cm] at (bot)
{$\mathbf{L}_1$ and $(X_{\mathbf{L}_1}, E)$}; 

\node[label,anchor=north] at (bot) {$0$};
\node[label,anchor=east] at (a) {$a$};
\node[label,anchor=west,xshift=-2pt] at (b) {$b$};
\node[label,anchor=east] at (c) {$c$};
\node[label,anchor=west,xshift=-1pt] at (d) {$d$};
\node[label,anchor=west,xshift=-2pt] at (e) {$e$};
\node[label,anchor=south] at (top) {$1$};
\end{scope}

\begin{scope}[xshift=12.5cm]
\node[] (ea) at (0,0) {$ea$};
\node[] (dc) at (0,1.25) {$dc$};
\node[] (de) at (0,2.5) {$de$};
\node[] (cb) at (0,3.75) {$cb$};
\path[thick,->] (ea.north) edge  (dc.south);
\path[thick,<->] (dc.north) edge  (de.south);
\path[thick,->] (cb.south) edge  (de.north);
\end{scope}

\begin{scope}[yshift=-11cm, xshift=8cm]
\node[unshaded] (bot) at (0,0) {};
\node[unshaded] (a) at (-1,2) {};
\node[unshaded] (b) at (0,2) {};
\node[unshaded] (c) at (1,2) {};
\node[unshaded] (d) at (-0.5,1) {};
\node[unshaded] (top) at (0,3) {}; 

\draw[order] (bot)--(d)--(a)--(top)--(b)--(d)--(bot)--(c)--(top);
\node[label,anchor=east,xshift=1pt] at (a) {$a$};
\node[label,anchor=east,xshift=1pt] at (b) {$b$};
\node[label,anchor=east,xshift=1pt] at (c) {$c$};
\node[label,anchor=east,xshift=1pt] at (d) {$d$};
\node[label,anchor=north] at (bot) {$0$};
\node[label,anchor=south] at (top) {$1$};
\node[label,anchor=north,yshift=-10pt, xshift=1.75cm] at (bot)
{$\mathbf{L}_5$ and $(X_{\mathbf{L}_5}, E)$}; 

\end{scope}

\begin{scope}[yshift=-11cm, xshift=12.5cm]
\node[] (ca) at (-0.75,0) {$ca$};
\node[] (cb) at (0.75,0) {$cb$};
\node[] (ba) at (-1.25,1.25) {$ba$};
\node[] (ab) at (1.25,1.25) {$ab$};
\node[] (dc) at (0,2.5) {$dc$};

\path[thick,->] (ba) edge  (dc);
\path[thick,->] (ab) edge  (dc);
\path[thick,->] (ca) edge  (ab);
\path[thick,->] (cb) edge  (ba);
\path[thick,<->] (ca) edge  (cb);
\path[thick,<->] (ca) edge  (ba);
\path[thick,<->] (cb) edge  (ab);
\end{scope}

\begin{scope}[yshift=-11cm]
\node[unshaded] (bot) at (0,0) {};
\node[unshaded] (a) at (1,1) {};
\node[unshaded] (b) at (0,1) {};
\node[unshaded] (c) at (-1,1) {};
\node[unshaded] (d) at (0.5,2) {};
\node[unshaded] (top) at (0,3) {};

\draw[order] (bot)--(b)--(d)--(top)--(d)--(a)--(bot)--(c)--(top);
\node[label,anchor=west,xshift=1pt] at (a) {$a$};
\node[label,anchor=west,xshift=1pt] at (b) {$b$};
\node[label,anchor=west,xshift=1pt] at (c) {$c$};
\node[label,anchor=west,xshift=1pt] at (d) {$d$};
\node[label,anchor=north] at (bot) {$0$};
\node[label,anchor=south] at (top) {$1$};
\node[label,anchor=north,yshift=-10pt, xshift=1.75cm] at (bot)
{$\mathbf{L}_4$ and $(X_{\mathbf{L}_4}, E)$}; 

\end{scope}
\begin{scope}[yshift=-11cm, xshift=4.5cm]
\node[] (ac) at (-0.75,0) {$ac$};
\node[] (bc) at (0.75,0) {$bc$};
\node[] (ab) at (-1.25,1.25) {$ab$};
\node[] (ba) at (1.25,1.25) {$ba$};
\node[] (cd) at (0,2.5) {$cd$};

\path[thick,<-] (ba) edge  (cd);
\path[thick,<-] (ab) edge  (cd);
\path[thick,<-] (ac) edge  (ba);
\path[thick,<-] (bc) edge  (ab);
\path[thick,<->] (ac) edge  (bc);
\path[thick,<->] (ac) edge  (ab);
\path[thick,<->] (bc) edge  (ba);
\end{scope}

\begin{scope}[yshift=-6cm, xshift=8cm]
\node[unshaded] (bot) at (0,0) {};
\node[unshaded] (e) at (-0.75,0.75) {};
\node[unshaded] (b) at (-1.5,1.5) {};
\node[unshaded] (c) at (0,1.5) {};
\node[unshaded] (d) at (1.5,1.5) {};
\node[unshaded] (a) at (0.75,2.25) {};
\node[unshaded] (top) at (0,3) {};
\draw[order] (bot)--(e)--(b)--(top)--(a)--(c)--(e)--(bot)--(d)--(a)--(top);
\node[label,anchor=west,xshift=1pt] at (a) {$a$};
\node[label,anchor=west,xshift=1pt] at (b) {$b$};
\node[label,anchor=west,xshift=1pt] at (c) {$c$};
\node[label,anchor=west,xshift=1pt] at (e) {$e$};
\node[label,anchor=west,xshift=1pt] at (d) {$d$};
\node[label,anchor=north] at (bot) {$0$};
\node[label,anchor=south] at (top) {$1$};
\node[label,anchor=north,yshift=-10pt, xshift=1.75cm] at (bot)
{$\mathbf{L}_3$ and $(X_{\mathbf{L}_3}, E)$};

\end{scope}

\begin{scope}[yshift=-6cm, xshift=12.5cm]
\node[] (ba) at (-0.75,0) {$ba$};
\node[] (ed) at (0.75,0) {$ed$};
\node[] (dc) at (-1.25,1.25) {$dc$};
\node[] (cb) at (1.25,1.25) {$cb$};
\node[] (db) at (0,2.5) {$db$};

\path[thick,->] (ba) edge  (ed);
\path[thick,->] (cb) edge  (ed);
\path[thick,->] (ba) edge  (dc);
\path[thick,->] (dc) edge  (cb);
\path[thick,<->] (dc) edge  (db);
\path[thick,<->] (db) edge  (cb);
\end{scope}

\end{tikzpicture}
\caption{Examples of finite (non-distributive) lattices and their dual digraphs. For clarity, we omit the loops whose existence follows from the reflexivity of $E$.}\label{fig:dual-examples}
\end{figure}

In case the lattice $\Lalg$ is finite, every filter is the up-set of a unique element and every ideal is the down-set of a unique element. Hence in a finite lattice $\Lalg$ we can represent every disjoint filter-ideal pair $\langle F, I \rangle$ by an ordered pair $\langle {\uparrow}x, {\downarrow}y\rangle$ where $x = \bigwedge F$ and $y= \bigvee I$. Thus for finite lattices we have $\MDFIP{x}{y} E \MDFIP{a}{b}$ if and only if $x \nleqslant b$. Examples of finite (non-distributive) lattices and their dual digraphs are presented in Figure~\ref{fig:dual-examples}. We denote by $xy$ the MDFIP $\MDFIP{x}{y}$ to make the labelling more compact. We remark that the directed edge set is not a transitive relation.

The properties of the digraphs dual to general lattices (with bounds) were described by Craig, Gouveia and Haviar~\cite{P3}. They were called \emph{TiRS graphs} there, yet in this paper (like in~\cite{P5} and~\cite{P7}) we prefer to use the terminology \emph{TiRS digraphs}. We recall that in the definition below $xE = \{\, y \in V \mid (x,y) \in E \,\}$ and $Ex = \{\, y \in V \mid (y,x) \in E  \,\}$. We also remark that the name \emph{TiRS} comes from combining the conditions (Ti), (R) and (S) below.

\begin{df}[{\cite[Definition 2.2]{P3}}]\label{def:TiRS} A TiRS digraph $G = (V,E)$ is a set $V$ and a
reflexive relation $E \subseteq V\times V$ such that:
\begin{itemize}
\item[(S)] If $x,y \in V$ and $x\neq y$ then $xE\neq yE$ or $Ex \neq Ey$.
\item[(R)] For all $x,y \in V$, if $xE \subset yE$ then $(x,y)\notin E$, and if $Ex \subset Ey$ then $(y,x)\notin E$.
\item[(Ti)] For all $x,y \in V$, if 
$(x,y)\in E$ then there exists $z \in V$ such that $zE \subseteq xE$ and $Ez \subseteq Ey$.
\end{itemize}
\end{df}

By \cite[Proposition 2.3]{P3}, for any general lattice $\Lalg$ (possibly infinite, with bounds), the dual digraphs $(P_\Lalg,E)$ and $(X_\Lalg,E)$ are TiRS digraphs.

Below is the first new result of this paper. It will be needed in Section~\ref{sec:ploscicaspaces}.

\begin{lem}\label{lem:antisym}
Let $G=(V,E)$ be a 
TiRS digraph with $x,y \in V$. If $xE \subseteq yE$ and $Ex \subseteq Ey$, then $x=y$. 
\end{lem}

\begin{proof}
Suppose that $xE\subseteq yE$ and $Ex \subseteq Ey$ and $x\neq y$. By (S), we must have $xE \subset yE$ or $Ex \subset Ey$. If $xE \subset yE$ then by (R) we get $(x,y)\notin E$, but with reflexivity this contradicts $Ex \subseteq Ey$. Similarly, applying (R) to $Ex \subset Ey$ would contradict $xE \subseteq yE$. 
\end{proof}

Plo\v{s}\v{c}ica endowed his dual $\mathcal{D}({\Lalg})$ of a lattice $\Lalg$ (having as a base set the set $P_{\Lalg} = \mph{\Lalg}{\twoB}$ of all MPHs from $\Lalg$ into $\twoB$) with the topology $\T_{\Lalg}$ having as a subbasis of closed sets all sets of the form
\[
V_a=\{\,f \in \mph{\Lalg}{\twoB} \mid f(a)=0 \,\} \quad \text{and} \quad
W_a=\{\,f \in \mph{\Lalg}{\twoB} \mid f(a)=1 \,\},
\]
where $a\in L$. 

One can check that 
$V_a\cap V_b = V_{a\vee b}$ 
and $W_a\cap W_b = W_{a\wedge b}$ for all $a,b\in L$. In the dual space $\mathcal{D}(\Lalg)= (P_{\Lalg}, E, \T_{\Lalg})$ of $\Lalg$ the topology $\T_{\Lalg}$ is T$_1$ and it is compact (it is the same topology as used by Urquhart~\cite[Lemma 6]{U78}). 

To recall facts concerning general digraphs $G =(X,E)$ from~\cite{Plos95}, let us now consider the two-element digraph $\twoT=(\{0,1\}, \le)$. We say that a~partial map $\varphi \colon X \to \twoT$ preserves the relation $E$ if $x, y \in \dom \varphi$ and $(x,y)\in E$ imply $\varphi(x)\leqslant \varphi(y)$. The (complete) lattice of maximal partial $E$-preserving maps from $G$ to $\twoT$ is denoted by $\mpe{G}{\twoT}$.

\begin{lem}
[cf.~{\cite[Lemma 1.3]{Plos95}}]\label{lem:Plos1.3}
Let $G=(X,E)$ be a digraph and let us consider
$\varphi \in \mpe{G}{\twoT}$. Then
\begin{enumerate}[{\upshape (i)}]
\item $\varphi^{-1}(0)=\{\, x\in X \mid \text{there is no}\ y\in\varphi^{-1}
(1)\ \text{with}\ (y,x)\in E \,\}$;
\item $\varphi^{-1}(1)=\{\, x\in X \mid
\text{there is no}\ y\in\varphi^{-1}
(0)\ \text{with}\ (x,y)\in E\,\}$.
\end{enumerate}

\end{lem}

The lemma above enables us to observe that for a digraph $G=(X,E)$ and $\varphi, \psi \in \mpe{G}{\twoT}$ we have
$$
\varphi^{-1}(1) \subseteq \psi^{-1}(1) \:\Longleftrightarrow\: \psi^{-1}(0) \subseteq \varphi^{-1}(0).
$$
It follows then that the reflexive and transitive binary relation $\leqslant$ defined on
$\mpe{G}{\twoT}$ by $\varphi\le\psi \iff \varphi^{-1}(1)\subseteq\psi^{-1}(1)$ is a partial order. 

Now we recall facts concerning general digraphs $\mathbb{P}= (X,E,\T)$ with topology from~\cite{P2}. A map $\varphi \colon (X_1,E_1,\T_1)\to (X_2,E_2,\T_2)$ between digraphs with topology is called a \emph{morphism} if it preserves the binary relation and is continuous as a map from  $(X_1,\T_1)$ to  $(X_2,\T_2)$. By a \emph{partial morphism} we mean a partial map $\varphi\colon  (X_1,E_1,\T_1)\to (X_2,E_2,\T_2)$ whose domain is a $\T_1$-closed subset of $X_1$ and the restriction of $\varphi$ to its domain is a morphism. A partial morphism is called a \emph{maximal partial morphism} (MPM), if there is no partial morphism properly extending it. For a digraph  with topology, $\mathbb{P}=(X,E,\T)$, we denote by $\mpm{\mathbb{P}}{\twoT_\T}$ the set of MPMs from $\mathbb{P}$ to the two-element digraph with the discrete topology, $\twoT_\T$. 

Plo\v s\v cica's representation theorem for general lattices (with bounds) can be presented as follows.

\begin{prop}
[{\cite[Lemmas 1.2, 1.5 and Theorem~1.7]{Plos95}}]\label{lem-eval}
Let\, $\Lalg$ be a  lattice (with bounds) and let\, $\mathcal{D}(\Lalg) = (P_{\Lalg}, E, \T_{\Lalg})$ be the dual of\, $\Lalg$.  For $a\in L$, let the evaluation map\, $e_a \colon \mathcal{D}(\Lalg) \to \twoT_\T$ be defined by
\[e_a(f)= \begin{cases}
                f(a) &\text{ $a\in \dom(f)$,}\\
                    - &\text{ undefined otherwise.}
\end{cases} \]
Then the following hold:
\begin{enumerate}[{\upshape (i)}]
\item The map\, $e_a$ is an element of $\mpm{\mathcal{D}(\Lalg)}{\twoT_\T}$ for each  $a \in L$.
\item Every\, $\varphi \in \mpm{\mathcal{D}(\Lalg)}{\twoT_\T}$ is of the form $e_a$ for some $a\in L$.
\item The map\, $e_{\Lalg}: \Lalg \to \mpm{\mathcal{D}(\Lalg)}{\twoT_\T}$ given by evaluation, $a \mapsto e_a$ {\upshape(}$a \in L${\upshape)}, is an isomorphism of\, $\Lalg$ onto the lattice\, $\mpm{\mathcal{D}(\Lalg)}{\twoT_\T}$, ordered by the relation\, $\varphi\le\psi$ if and only if\, $\varphi^{-1}(1)\subseteq\psi^{-1}(1)$.
\end{enumerate}
\end{prop}

Now we recall some facts from~\cite{P2} that will also be useful in the next section. For a digraph
$G=(X,E)$ one can consider the triple (called a \emph{context}) $\mathbb{K}(G) := (X,X,E^\complement)$, where the relation $E^\complement \subseteq X\times X$ is the complement of the digraph relation $E$: $E^\complement = (X\times X){\setminus} E$. One can then define a Galois connection via so-called \emph{polars} as follows. The maps
$$E^\complement_\triangleright : (\powerset(X),\subseteq) \to (\powerset(X),\supseteq)\quad \text{and}\quad E^\complement_\triangleleft : (\powerset(X),\supseteq) \to (\powerset(X),\subseteq)$$
are given by
\begin{align}
E^\complement_\triangleright (Y) &=\{\, x \in X \mid (\forall\, y \in Y) (y,x) \notin E \,\},\notag\\
E^\complement_\triangleleft (Y) &=\{\, z \in X \mid (\forall\, y \in Y) (z,y) \notin E \,\}.\notag
\end{align}
The so-called \emph{concept lattice} $\mathrm{CL}(\mathbb{K}(G))$  of the
context $\mathbb{K}(G)= (X,X,E^\complement)$,
given by
$$ \mathrm{CL}(\mathbb{K}(G)) = \{\, Y \subseteq X \mid
(E^\complement_\triangleleft \circ E^\complement_\triangleright) (Y)= Y \,\},$$
is a complete lattice when ordered by inclusion.

The lemma below is needed in the proof of Lemma~\ref{lem:Lemma2-equiv}. 

\begin{lem}\label{lem:Ex-xE-phi}
Let $G=(X,E)$ be a reflexive digraph and $\varphi \in  \mpe{G}{\twoT}$.

\begin{enumerate}[{\upshape (i)}]
\item If $x \in \varphi^{-1}(1)$ and $zE \subseteq xE$, then $z \in \varphi^{-1}(1)$.  
\item If $x \in \varphi^{-1}(0)$ and $Ez \subseteq Ex$, then $z \in \varphi^{-1}(0)$. 
\end{enumerate}
\end{lem}
\begin{proof} For (i), let $x \in \varphi^{-1}(1)$ and $zE \subseteq xE$. Consider any $y \in \varphi^{-1}(0)$. By Lemma~\ref{lem:Plos1.3}, $(x,y) \notin E$, so  $ y\notin xE$. From the assumption we get $y \notin zE$, i.e. $(z,y) \notin E$. As $y \in \varphi^{-1}(0)$ was arbitrary,  by Lemma~\ref{lem:Plos1.3} we get $\varphi(z)=1$. The proof of (ii) follows by a dual argument. 
\end{proof}

The following result will be needed in the next section.

\begin{lem}\label{lem:Lemma2-equiv}
Let $G=(X,E)$ be a TiRS digraph and $\varphi, \psi \in 
\mpe{G}{\twoT}$. 
\begin{enumerate}[{\upshape (i)}]
    \item If $\varphi^{-1}(0) \subseteq X{\setminus}\psi^{-1}(1)$ then $\varphi^{-1}(0) \subseteq \psi^{-1}(0)$.
    \item If $\varphi^{-1}(1) \subseteq X{\setminus}\psi^{-1}(0)$ then $\varphi^{-1}(1)\subseteq \psi^{-1}(1)$.
\end{enumerate}
\end{lem}

\begin{proof}
We will prove the contrapositive of (i). Assume there exists $y\in \varphi^{-1}(0)$ such that $y \notin \psi^{-1}(0)$. Then by Lemma~\ref{lem:Plos1.3} there exists $x \in \psi^{-1}(1)$ with $(x,y) \in E$. By the condition (Ti), there exists $z$ with $zE \subseteq xE$ and $Ez \subseteq Ey$. By Lemma~\ref{lem:Ex-xE-phi} we get that $z \in \varphi^{-1}(0)$ and also $z \in \psi^{-1}(1)$. Hence $\varphi^{-1}(0) \nsubseteq X{\setminus}\psi^{-1}(1)$. The proof of (ii) 
follows a similar argument. 
\end{proof}

\section{Plo\v{s}\v{c}ica spaces}
\label{sec:ploscicaspaces}

In this section we provide a description of the Urquhart dual spaces of general lattices (with bounds) in Plo\v{s}\v{c}ica's setting. This is the first of two unfulfilled (in our view) tasks in the paper~\cite{Plos95}. 
We will naturally call these objects \emph{Plo\v{s}\v{c}ica spaces}. These spaces are TiRS digraphs with topology $\mathbb{P}=(X,E,\T)$, where $E$ is the edge set and $\T$ is the same compact topology that was used by both 
Urquhart~\cite{U78} and Plo\v{s}\v{c}ica~\cite{Plos95}. 
We believe that Plo\v{s}\v{c}ica spaces are easier to work with than Urquhart's L-spaces equipped with two quasi-order relations. Also this concept works naturally with the set $\mpm{\mathbb{P}}{\twoT_\T}$ of MPMs, which forms the original lattice $\Lalg$ in case the 
Plo\v{s}\v{c}ica space
$\mathbb{P}$ is the dual $\mathcal{D}(\Lalg) = (P_{\Lalg}, E, \T)$ of\, $\Lalg$. Lastly, in Plo\v{s}\v{c}ica spaces the edge relation becomes the Priestley order on $\Lalg$ in case the lattice $\Lalg$ is distributive.  

In this section we also prove that every Plo\v{s}\v{c}ica space 
$\mathbb{P}=(X,E,\T)$ is 
digraph-homeomorphic to 
the space $\mathbb{P}_{\mathcal{L}(\mathbb{P})}=(P_{\mathcal{L}(\mathbb{P})},E,\T_{\mathcal{L}(\mathbb{P})})$ 
dual to the lattice $\mathcal{L}(\mathbb{P})$. This is, in the new setting, an equivalent version of Urquhart's result~\cite[Theorem~2]{U78} that every L-space $\mathbb{S}=(X, \le_1, \le_2, \T)$ is 
order-homeomorphic to 
the dual space of the lattice $\mathcal{L}(\mathbb{S})$. 

\begin{df}\label{def:ploscicaspaces}
A \emph{Plo\v{s}\v{c}ica space} is a structure $\mathbb{P}=(X,E,\T)$ such that
\begin{enumerate}[(1)]
 \item $(X,E)$ 
 is a TiRS digraph and $(X,\T)$ is a compact topological space.
 \item $\mathbb{P}$ is totally edge-disconnected, meaning that for all $x,y\in X$ such that $(x,y)\notin E$ there exists $\varphi\in\mpm{\mathbb{P}}{\twoT_\T}$ with $\varphi(x)=1$ and $\varphi(y)=0$.
\item For any $\varphi$, $\psi \in \mpm{\mathbb{P}}{\twoT_\T}$, the sets
$$ E_{\triangleleft}^{\complement}(\varphi^{-1}(0) \cap \psi^{-1}(0)) \quad \text{ and } \quad E_{\triangleright}^{\complement}(\varphi^{-1}(1)\cap \psi^{-1}(1))$$
are closed.
\item The family
 $$ \left\{\, X{\setminus}\varphi^{-1}(1)\mid \varphi \in \mpm{\mathbb{P}}{\twoT_\T}  
\,\right\} \cup \left\{\, X{\setminus}\varphi^{-1}(0) \mid \varphi\in \mpm{\mathbb{P}}{\twoT_\T}  
\,\right\} $$
forms a subbase for $\T$.
\end{enumerate}

\end{df}

The total edge-disconnectedness of $\mathbb{P}$ has the following consequence: 

\begin{lem}\label{lem:up/down set separation}
    Let $\mathbb{P}=(X,E,\T)$  be a digraph with topology that is totally edge-disconnected. Then for any $x, y\in X$ we have: 
    \begin{enumerate}[{\upshape (i)}]
        \item If $yE \not\subseteq xE $ then there exists $\varphi \in \mpm{\mathbb{P}}{\twoT_\T}$ such that $\varphi(x) = 1$ and $\varphi(y)\neq 1$. 
        \item If $Ey \not\subseteq Ex $ then there exists $\varphi \in \mpm{\mathbb{P}}{\twoT_\T}$ such that $\varphi(x) = 1$ and $\varphi(y)\neq 1$. 
    \end{enumerate}
\end{lem}

\begin{proof} 
    If $yE\not\subseteq xE$ then there is $z\in X$ such that $(y,z)\in E$ and $(x,z)\notin E$. By the total edge-disconnectedness property there exists $\varphi\in \mpm{\mathbb{P}}{\twoT_\T}$ such that $\varphi(x)=1$ and $\varphi(z)=0$. Since $(y,z)\in E$ and $\varphi(z)=0$, we have $\varphi(y) \neq 1$. This gives (i). 
    We can apply a similar argument to show (ii). 
\end{proof}

The total edge-disconnectedness property (2) 
in the definition of Plo\v{s}\v{c}ica spaces is a generalisation of the total order-disconnectedness of Priestley spaces (cf. \cite{Pr70}, \cite{Pr72}). We recall that total order-disconnectedness means that for any two points $x \ne y$ there exists a clopen up-set (or, equivalently, a clopen down-set) that separates them. By Lemma~\ref{lem:antisym} we have in TiRS digraphs that if $x\neq y$ then $yE\not\subseteq xE$ or $Ey\not\subseteq Ex$. Hence by the above lemma, the total edge-disconnectedness can be thought of as saying that if $x \neq y$ then there is an MPM 
for which $\varphi(x)=1$ and $\varphi(y)\neq 1$ or there is an MPM $\psi \in \mpm{\mathbb{P}}{\twoT_\T}$ for which $\psi(x)=0$ and $\psi(y) \neq 0$.
(If $\Lalg$ is distributive and $E$ is the Priestley order on its dual, this gives a clopen up-set $\varphi^{-1}(1)$ --- or, equivalently, a clopen down-set $\psi^{-1}(0)$ --- that separates $x$ and $y$.)

We recall that in Plo\v{s}\v{c}ica's representation, the dual space of a lattice $\Lalg$ is given by 
$\mathcal{D}(\Lalg)=(P_{\Lalg},E,\T_{\Lalg})$, where $P_{\Lalg} = \mph{\Lalg}{\twoB}$ and $\T_\Lalg$ has a subbase of open sets given by
$\left\{\, P_{\Lalg} {\setminus}e_a^{-1}(1)\mid a \in L  \,\right\} \cup \left\{\, 
P_{\Lalg} {\setminus}e_a^{-1}(0) \mid a \in L   \,\right\}$. Note that this corresponds to the subbase $\{\,X_\Lalg{\setminus}u(a) \mid a \in L\,\} \cup \{\, X_\Lalg {\setminus} r(u(a)) \mid a \in L \,\}$ as defined by Urquhart~\cite[p.~47]{U78}. (The maps $u$ and $r$ are also defined there.) 

\begin{rem}
After the submission of this paper, we were made aware of another definition of Plo\v{s}\v{c}ica space in~\cite[Section 11]{BCM25}. We observe that our definitions are both translations of Urquhart's dual spaces and hence are essentially the same. Core to our definition is the use of MPMs to define the closure operator required to recover the lattice---to maintain similarity with Priestley duality. In~\cite{BCM25} the authors emulate ideas from modal logic and use a $\Box$ and  $\blacklozenge$ to define the closure operator applied in the recovery of the lattice and the description of the topology. 
\end{rem}

\begin{lem}\label{lem:doubly-disconnected}
Let \,$\Lalg$ be a bounded lattice. Then the structure $\mathcal{D}(\Lalg)=(P_\Lalg,E,\T_\Lalg)$ is totally edge-disconnected. 
\end{lem}

\begin{proof}
    If $x,y\in P_\Lalg$ and $(x, y)\notin E$ then by definition of $E$ we have that there exists $a \in \dom x\cap \dom y$ such that $x(a)\not\leqslant y(a)$. It must then be the case that $x(a) = 1$ and $y(a) = 0$. Therefore taking $e_a \in \mpm{\mathcal{D}}{\twoT_\T}$ gives the result. 
\end{proof}

The proof of the following lemma is the same as that for~\cite[Lemma 6]{U78}. 

\begin{lem}\label{lem:compactP(L)}
Let $\Lalg$ be a lattice with bounds. Then $(P_\Lalg,\T_\Lalg)$ is a compact space.
\end{lem}

\begin{rem}
We remark that Urquhart~\cite[Section 2]{U78} simply says that for an L-space $\mathbb{S}$ the family of all doubly closed stable sets ordered by inclusion forms the dual lattice of $\mathbb{S}$ with the lattice operations given by $Y\wedge Z = Y\cap Z$ and $Y\vee Z = \ell (r(Y)\cap r(Z))$. (The maps $\ell$ and $r$ are defined in \cite[p.~46]{U78}.)

Urquhart's doubly closed stable sets $Y,Z$ correspond to $\varphi^{-1}(1)$ and $\psi^{-1}(1)$ for our MPMs $\varphi, \psi \in \mpm{\mathbb{S}}{\twoT_\T}$. 
Hence his $Y\wedge Z = Y\cap Z$ corresponds to our $(\varphi \wedge \psi)^{-1}(1)$ and $Y\vee Z = \ell (r(Y)\cap r(Z))$  corresponds to our $(\varphi \vee \psi)^{-1}(0)$. Urquhart's condition (2) on an L-space says that for doubly closed stable sets $Y,Z$, the sets $r(Y\cap Z)$ (corresponding to our set $(\varphi \wedge \psi)^{-1}(0)= E_{\triangleright}^{\complement}
\left(\varphi^{-1}(1)
\cap \psi^{-1}(1)\right)$) and $ \ell (r(Y)\cap r(Z))$ (corresponding to our set $(\varphi \vee \psi)^{-1}(1)=E^{\complement}_{\triangleleft}\left(
\varphi^{-1}(0)\cap \psi^{-1}(0) \right)$) are closed.
\end{rem}

We now prove a series of  results that will culminate in Theorem~\ref{thm:graph-homeomorphic}, which shows that every Plo\v{s}\v{c}ica space arises as the dual space of a lattice with bounds.

\begin{prop}
Let $\Lalg$ be a lattice with bounds. Then $\mathcal{D}(\Lalg) = (P_\Lalg,E,\T_\Lalg)$ is a Plo\v{s}\v{c}ica space.
\end{prop}

\begin{proof}
As mentioned after Definition~\ref{def:TiRS}, the fact that $(P_\Lalg,E)$ is a TiRS digraph follows from~\cite[Proposition 2.3]{P3}. The total edge-disconnectedness of $\mathcal{D}(\Lalg)$ follows from Lemma~\ref{lem:doubly-disconnected} and the compactness of the space $(P_\Lalg,\T_\Lalg)$ follows from Lemma~\ref{lem:compactP(L)}. 

Condition (4) simply follows from the Plo\v{s}\v{c}ica representation theorem, i.e. the fact that all elements of $\mpm{\mathcal{D}(\Lalg)}{\twoT_\T}$ are of the form $e_a$ for some $a \in L$ (cf. Proposition~\ref{lem-eval}(ii)). To prove condition (3), we use the fact $\varphi=e_a$ and $\psi=e_b$ for some $a, b \in L$, and $\varphi \wedge \psi = e_{a\wedge b}$ and $\varphi \vee \psi = e_{a \vee b}$.
\end{proof}

\begin{prop}\label{prop:L(S)boundedlattice}
Let $\mathbb{P}=(X,E,\T)$ be a Plo\v{s}\v{c}ica space. Let $\leqslant$ be the ordering on $\mpm{\mathbb{P}}{\twoT_\T}$ defined by
\[\varphi \leqslant \psi \quad \Longleftrightarrow \quad
\varphi^{-1}(1) \subseteq \psi^{-1}(1)
\quad \Longleftrightarrow \quad
\psi^{-1}(0) \subseteq \varphi^{-1}(0).\]
Then $\mathcal{L}(\mathbb{P})=(\mpm{\mathbb{P}}{\twoT_\T},\leqslant)$ is a bounded lattice.
\end{prop}

\begin{proof}
Clearly the constant maps $\varphi_1(x)=1$ and $\varphi_0(x)=0$ are MPMs and are, respectively, the greatest and least element of $\mathcal{L}(\mathbb{P})$.

For $\varphi, \psi \in \mpm{\mathbb{P}}{\twoT_\T}$ we define the maps
\[M_{\varphi,\psi}(x) =
\begin{cases} 1\quad &\text{if }\ x\in \varphi^{-1}(1)\cap \psi^{-1}(1),\\
             0\quad &\text{if } x \in E_{\triangleright}^{\complement}\left(\varphi^{-1}(1)\cap \psi^{-1}(1)\right)
\end{cases}\]
and
$$
J_{\varphi,\psi}(x) =
\begin{cases} 1\quad &\text{if }
x \in E^{\complement}_{\triangleleft}\left(
\varphi^{-1}(0)\cap \psi^{-1}(0) \right),
\\
             0\quad &\text{if }
x \in \varphi^{-1}(0)\cap \psi^{-1}(0).
\end{cases}
$$
The maps $M_{\varphi,\psi}$ and $J_{\varphi,\psi}$ are defined such that it can easily be shown that they preserve the relation 
$E$. The fact they are continuous, i.e. their domains are closed, is guaranteed by conditions (3) and (4) of Definition~\ref{def:ploscicaspaces}. Indeed, (3) guarantees that $M_{\varphi,\psi}^{-1}(0)$ and $J_{\varphi,\psi}^{-1}(1)$ are closed, while (4) yields that $M_{\varphi,\psi}^{-1}(1)$ and $J_{\varphi,\psi}^{-1}(0)$ are closed. Hence $M_{\varphi,\psi}(x)$ and $J_{\varphi,\psi}(x)$ are elements of $\mpm{\mathbb{P}}{\twoT_\T}$.

By the definition of the ordering  on $\mpm{\mathbb{P}}{\twoT_\T}$ it is clear that $M_{\varphi,\psi}$ and $J_{\varphi,\psi}$ are the greatest lower bound and the least upper bound for $\varphi$ and  $\psi$, respectively. So $M_{\varphi,\psi} = \varphi \wedge \psi$ and $J_{\varphi,\psi} = \varphi \vee \psi$.
\end{proof}

For the theorem below, see \cite[Theorem 1]{U78} or \cite[Theorem 1.7]{Plos95}.

\begin{thm}\label{thm:rep-for-L}
Let $\Lalg$ be a bounded lattice. Then $\Lalg \cong \mathcal{L}(\mathcal{D}(\Lalg))$.
\end{thm}

\begin{proof}
We show that the map $\nu: \Lalg \to \mathcal{L}(\mathcal{D}(\Lalg))$ given by $\nu (a)= e_a$ is an isomorphism, where the evaluation map $e_a$ was defined in Proposition~\ref{lem-eval}. If we have $a \nleqslant b$, then the partial homomorphism $f$ with $f^{-1}(1)={\uparrow}a$ and $f^{-1}(0)={\downarrow} b$ can be extended, by Zorn's Lemma, to an MPH $\bar{f}$. Then $e_a(\bar{f})=1\nleqslant 0 = e_b(\bar{f})$, hence $e_a\nleqslant e_b$. Conversely, if $e_a\nleqslant e_b$, then there exists some MPH $g$ such that  $e_a(g)=1$ and $e_b(g)=0$. Hence $a\nleqslant b$. 

The fact that every MPM $\varphi \in \mpm{\mathcal{D}(\Lalg)}{\twoT_\T}$ is of the form $e_a$ (i.e. that $\nu$ is onto)  follows from Proposition~\ref{lem-eval}(ii). 
\end{proof}

For two Plo\v{s}\v{c}ica spaces $\mathbb{P}_1$ and $\mathbb{P}_2$, we write $\mathbb{P}_1 \cong \mathbb{P}_2$ to indicate they are \emph{digraph-homeomorphic} to one another. That is, there exists $\vartheta: X_1 \to X_2$ such that $xE_1y$ iff $\vartheta(x)E_2 \vartheta(y)$ and $\vartheta$ is a homeomorphism. The lemma below defines such a $\vartheta$ from $\mathbb{P}$ to $\mathcal{D}(\mathcal{L}(\mathbb{P}))$.

\begin{lem}\label{lem:defn-v}
Let $\mathbb{P}=(X,E,\T)$ be a Plo\v{s}\v{c}ica space. For $x \in X$, define a partial map $\varepsilon_x$ from $\mathcal{L}(\mathbb{P})$ to $\twoB$ such that for $\varphi \in \mpm{\mathbb{P}}{\twoT_\T}$
\[\varex(\varphi) = \begin{cases}
    \varphi(x) & \text{ if } x \in\dom\varphi, \\
    - & \text{otherwise.}
\end{cases}\]

\noindent 
Then $\varex \in \mph{\mathcal{L}(\mathbb{P})}{\twoB}$.
\end{lem}

\begin{proof}
To show that $\varepsilon_x$ is a partial homomorphism, it suffices to prove that $\varex^{-1}(0)$ is an ideal and $\varex^{-1}(1)$ is a filter. 
We will show only that $\varex^{-1}(0)$ is an ideal, the fact that $\varex^{-1}(1)$ is a filter will follow by a dual argument. Let $\varphi \in \varex^{-1}(0)$ and $\psi \leqslant \varphi$. Then $\varphi(x)=0$ and since $\varphi^{-1}(0)\subseteq \psi^{-1}(0)$, we get $\psi(x)=0$ so $\psi \in \varex^{-1}(0)$. Now let $\varphi,\psi \in \varex^{-1}(0)$. Then $\varphi(x)=\psi(x)=0$, so $x \in \varphi^{-1}(0)\cap \psi^{-1}(0)=(\varphi \vee \psi)^{-1}(0)$, thus $\varphi \vee \psi \in \varex^{-1}(0)$. Hence $\varex$ is a partial homomorphism.

Now we show that the domain of $\varex$ is maximal. Suppose there exists a filter $F \subseteq \mpm{\mathbb{P}}{\twoT_\T}$ properly containing $\varex^{-1}(1)$ and disjoint from $\varex^{-1}(0)$. We notice that $\varex^{-1}(1) = \{\varphi \in  \mpm{\mathbb{P}}{\twoT_\T} \mid \varphi(x)=1\}$.
For the rest of the proof, we let $Y=\bigcap \{ \psi^{-1}(1) \mid \psi \in F \} $. 
Then we have 
$$ Y \subseteq \bigcap \{\varphi^{-1}(1) \mid \varphi(x)=1\}$$ 
and there exists $\psi \in F$ with $\psi(x) \neq 1$. Hence $x\notin Y$. Let us fix an element $z\in Y$. We claim that then $zE \subseteq xE$. 

If $zE \not\subseteq xE$, then by Lemma~\ref{lem:up/down set separation} there exists $\varphi \in \mpm{\mathbb{P}}{\twoT_\T}$ such that $\varphi(x)=1$ and $\varphi(z)\neq 1$, a contradiction.

We have that $z\ne x$ since $x\notin Y$. By Lemma~\ref{lem:antisym} we get $Ez \nsubseteq Ex$ and by the total edge-disconnectedness and Lemma~\ref{lem:up/down set separation} there exists $\psi_z\in\mpm{\mathbb{P}}{\twoT_\T}$ such that $\psi_z(x)=0$ and $\psi_z(z)\neq 0$, thus $z\in X{\setminus}  \psi_z^{-1}(0)$.

Notice now that we have shown above
$$ Y \subseteq \bigcup \left\{\, X{\setminus}\psi_z^{-1}(0) 
\mid z\in Y 
\,\right\}. $$
By applying the compactness of $\mathbb{P}$, we obtain
$$ Y \subseteq \bigcup_{i=1}^n X{\setminus} \psi_{z_i}^{-1}(0) \quad (\ast)$$
for some elements $z_1,\dots,z_n \in Y$. 

Let us define $\overline{\psi}:= \bigvee \{ \psi_{z_i} \mid 1 \leq i \leq n \}$. Consider taking complements of the set containment in $(\ast)$. Then, by applying the definition of the join, we obtain
$$
\overline{\psi}^{-1}(0) = \bigcap_{i=1}^n \psi_{z_i}^{-1}(0) \subseteq \bigcup \{X \setminus \psi^{-1}(1) \mid \psi \in F \,\}. \quad (\ast\ast)
$$
By applying the compactness of $\mathbb{P}$ again, we get for some $m$:
$$
\overline{\psi}^{-1}(0) \subseteq \bigcup_{j=1}^m \left\{\,X \setminus \psi_j^{-1}(1) \mid \psi_j \in F \,\right\} = X \setminus \bigcap_{j=1}^m \left\{\,\psi_j^{-1}(1) \mid \psi_j \in F \,\right\}.$$
We now define $\hat{\psi}:= \bigwedge \{ \psi_j \mid 1 \leq j \leq m\}$. Then $\overline{\psi}^{-1}(0) \subseteq X \setminus \hat{\psi}^{-1}(1)$. 
By Lemma~\ref{lem:Lemma2-equiv}(i) we get that  ${\overline{\psi}}^{-1}(0) \subseteq {\hat{\psi}}^{-1}(0)$. By the definition of the order in the lattice $\mathcal{L}(\mathbb{P})$ it follows that $ \hat{\psi}\leq \overline{\psi} $.

Since $F$ is a filter in the lattice $\mathcal{L}(\mathbb{P})$ and $\psi_1,\dots,\psi_m\in F$, we have $\hat{\psi} \in F$. It follows that $\overline{\psi} \in F$. Now by $(\ast\ast)$ and the fact that 
$\psi_{z_i}(x)=0$ for all $1 \leq i \leq n$,  we get  $\overline{\psi}(x)=0$. 
It follows that $\overline{\psi} \in F \cap \varex^{-1}(0)$, which contradicts that $F$ and $\varex^{-1}(0)$ are disjoint.

We have shown that there is no filter $F$ properly containing $\varex^{-1}(1)$ and disjoint with $\varex^{-1}(0)$. One can show similarly that there is no ideal properly containing $\varex^{-1}(0)$ and disjoint with $\varex^{-1}(1)$. Hence $\mdfip{\varex^{-1}(1)}{\varex^{-1}(0)}$ 
is an MDFIP, proving the maximality of the partial homomorphism $\varex$.
\end{proof}

Lemma~\ref{lem:defn-v} allows us to define $\vartheta : \mathbb{P} \to \mathcal{D} (\mathcal{L}(\mathbb{P}))$ by $\vartheta (x)=\varepsilon_x$. We use the map $\vartheta$ to show that a Plo\v{s}\v{c}ica space is digraph-homeomorphic to its second dual.

\begin{thm}\label{thm:graph-homeomorphic}
Let $\mathbb{P}=(X,E,\T)$ be a Plo\v{s}\v{c}ica space. Then
$\mathbb{P} \cong \mathcal{D} (\mathcal{L}(\mathbb{P}))$.
\end{thm}

\begin{proof}
To show that $\vartheta$ is a digraph homeomorphism we show the following:
\begin{enumerate}[{\upshape (i)}]
\item For all $x,y \in X$,  
$(x,y) \in E$ iff $(\varex,\varepsilon_y) \in E$. 
\item If $x \neq y$ then $\varex \neq \varepsilon_y$. 
\item For all $f \in \mph{\mathcal{L}(\mathbb{P})}{\twoB}$ there exists $x \in X$ such that $\varex=f$. 
\item $\vartheta$ and $\vartheta^{-1}$ are continuous. 
\end{enumerate}

If $(x,y)\in E$, then by Lemma~\ref{lem:Plos1.3} we have
\begin{align*}
&(\forall \varphi \in \mathcal{L}(\mathbb{P}))(\varphi(x)=1 \Longrightarrow \varphi(y) \neq 0)\\ \Longleftrightarrow \quad &(\forall \varphi \in \mathcal{L}(\mathbb{P}))(\varex(\varphi)=1 \Longrightarrow \varepsilon_y(\varphi) \neq 0)\\
 \Longleftrightarrow \quad
&\varex^{-1}(1) \cap \varepsilon_y^{-1}(0) = \emptyset \\
 \Longleftrightarrow \quad & (\varex,\varepsilon_y) \in E.
\end{align*}
For the converse, assume $(x,y) \notin E$. By the total edge-disconnectedness of~$\mathbb{P}$ (Definition~\ref{def:ploscicaspaces}(2)) there exists $\varphi \in \mpm{\mathbb{P}}{\twoT_\T}$ such that $\varphi(x)=1$ and $\varphi(y)=0$. It follows that $\varphi \in \varepsilon_x^{-1}(1)\cap \varepsilon_y^{-1}(0)$ and thus $(\varepsilon_x, \varepsilon_y)\notin E$. 

Next, we show (ii). If $x\neq y$, then by the total edge-disconnectedness of $\mathbb{P}$ there exists $\varphi \in \mpm{\mathbb{P}}{\twoT_\T}$ such that $\varphi(x)\neq \varphi(y)$. Hence $\varex(\varphi)\neq \varepsilon_y(\varphi)$, so $\varex \neq \varepsilon_y$.

For (iii), we let $f \in \mph{\mathcal{L}(\mathbb{P})}{\twoB}$. Consider
$$ \mathcal{F} =\{\, \varphi^{-1}(1) \mid \varphi \in f^{-1}(1)\,\} \cup \{\, \varphi^{-1}(0) \mid \varphi \in f^{-1}(0)\,\}. $$
We claim that $\mathcal{F}$ has the Finite Intersection Property (FIP). Notice that for $I, J$ finite, we have
$$ \bigcap_{i \in I} \varphi^{-1}(1) = \left( \bigwedge_{i \in I} \varphi_i\right)^{-1}(1) \:\text{ and } \:\bigcap_{j\in J} \varphi_j^{-1}(0)=\left( \bigvee_{j \in J} \varphi_j\right)^{-1}(0)$$
and hence testing FIP can be reduced to testing $\varphi^{-1}(1) \cap \psi^{-1}(0)$ for some $\varphi \in f^{-1}(1)$ and some $\psi\in f^{-1}(0)$. If for such $\varphi, \psi$ we have $\varphi^{-1}(1) \cap \psi^{-1}(0)=\emptyset$ then $\varphi^{-1}(1)\subseteq X{\setminus}\psi^{-1}(0)$. By Lemma~\ref{lem:Lemma2-equiv}(ii) we get $\varphi^{-1}(1) \subseteq \psi^{-1}(1)$ and so $\psi \in f^{-1}(1)$, a contradiction.

We need $x \in X$ such that $\varex =f$. Since $\mathcal{F}\subseteq \powerset(X)$ and it has the FIP, it can be extended to an ultrafilter $\mathcal{U}$ on $\powerset(X)$. Since $(X,\tau)$ is compact we know that $\mathcal{U}$ must converge to some point, say  $x \in X$. Now $x\in \varphi^{-1}(1)$ for all $\varphi \in f^{-1}(1)$, and $x \in\varphi^{-1}(0)$ for all $\varphi \in f^{-1}(0)$. 
Hence $f^{-1}(1) \subseteq \varex^{-1}(1)$ and $f^{-1}(0) \subseteq \varex^{-1}(0)$. The equality then follows from the maximality of $f$.

To see that $\vartheta$ is continuous, consider $\vartheta^{-1}: \mathcal{D} (\mathcal{L}(\mathbb{P})) \to \mathbb{P}$. Let $V_\varphi$ be a subbasic closed set of $\mathcal{D}(\mathcal{L}(\mathbb{P}))$, say $V_\varphi =\{\, f \in \mph{\mathcal{L}(\mathbb{P})}{\twoB} \mid f(\varphi)=0\}$. We have $V_\varphi=\{\, \varepsilon_x \mid \varepsilon_x(\varphi)=0\,\}=\{\,\varepsilon_x \mid \varphi(x)=0\, \}$ and hence we obtain $\vartheta^{-1}(V_\varphi)=\{\,x\in X\mid \varphi(x)=0\,\}=\varphi^{-1}(0)$, a subbasic closed set of $\mathbb{P}$. A~similar calculation shows that $\vartheta^{-1}(W_\varphi)=\varphi^{-1}(1)$.

For the continuity of $\vartheta^{-1}$, consider a subbasic closed set $\varphi^{-1}(1)$ where 
$\varphi \in~\mpm{\mathbb{P}}{\twoT_\T}$. The following shows that $\vartheta$ maps $\varphi^{-1}(1)$ to a subbasic closed set: 
\begin{align*}
\vartheta(\varphi^{-1}(1))  &= \{\, \varepsilon_x \mid x \in \varphi^{-1}(1) \,\} \\
& = \{\, \varepsilon_x \mid \varphi(x)= 1 \,\}\\
& = \{\, \varepsilon_x \mid \varepsilon_x(\varphi)= 1 \,\}\\
& = \{ \, f \in \mph{\mathcal{L}(\mathbb{P})}{\twoB} \mid f(\varphi)=1\,\}\\
& = \{ \, f \in \mph{\mathcal{L}(\mathbb{P})}{\twoB} \mid e_\varphi(f)=1\,\} = e_\varphi^{-1}(1).
\end{align*}
\end{proof}

\section{Dual spaces of semidistributive lattices}
\label{sec:dualspacesSDL}

A lattice $\Lalg$ is \emph{join-semidistributive}
if it satisfies the following quasi-equation for all $a,b,c \in L$ (cf. \cite{AN-Ch3}): 
\begin{itemize}
\item[(JSD)] $a\vee b = a\vee c \quad \Longrightarrow \quad a \vee b = a \vee (b \wedge  c).$
\end{itemize}
Dually,
$\Lalg$ is \emph{meet-semidistributive} if it satisfies:
\begin{itemize}
\item[(MSD)] $a\wedge  b = 	a\wedge c \quad \Longrightarrow \quad a \wedge b = a \wedge (b \vee c).$
\end{itemize}
A lattice is \emph{semidistributive} if it satisfies both (JSD) and (MSD). 

We recall results from~\cite{P5} describing finite join- and meet-semi\-distri\-bu\-ti\-ve lattices and their dual TiRS digraphs.

\begin{thm}[{\cite[Theorem 3.2]{P5}}]\label{thm:SD-DFIP}
Let $\Lalg$ be a finite lattice.
\begin{itemize}
\item[{\upshape (i)}] $\Lalg$ is not join-semidistributive if and only if there exist distinct maximal disjoint filter-ideal pairs of the form $\langle {\uparrow}b, {\downarrow}a \rangle$ and $\langle {\uparrow}c, {\downarrow}a\rangle$ for some $a,b,c  \in L$.
\item[{\upshape (ii)}] $\Lalg$ is not meet-semidistributive if and only if there exist
distinct maximal disjoint filter-ideal pairs of the form $\langle {\uparrow}a, {\downarrow}b \rangle$ and $\langle {\uparrow}a, {\downarrow}c\rangle$ for some $a,b,c  \in L$.
\end{itemize}
\end{thm}

The theorem below gives a characterisation of the dual digraphs  of finite join- and meet-semidistributive lattices. It is important to observe that each of the conditions (i), (ii) and (iii) below is a strengthening of the (S) condition from the definition of TiRS digraphs (Definition~\ref{def:TiRS}).

\begin{thm}[{\cite[Theorem 3.6]{P5}}]\label{SD:char}
Let $G=(V,E)$ be a finite TiRS digraph. Then
\begin{enumerate}[{\upshape (i)}] 
\item $G$ is the dual digraph of a finite lattice satisfying {\upshape (JSD)} if and only if it satisfies the following condition: 
$$ \text{\upshape (dJSD)} \qquad (\forall u,v\in V)\ \ u \neq v \quad \Longrightarrow \quad Eu \neq Ev.
$$
\item $G$ is the dual digraph of a finite lattice satisfying {\upshape (MSD)} if and only if it satisfies the following condition: 
$$\text{\upshape (dMSD)} \qquad (\forall u,v\in V)\ \ u \neq v \quad \Longrightarrow \quad uE \neq vE. $$
\item $G$ is the dual digraph of a finite semidistributive lattice if and only if it satisfies the following condition: 
$$\text{\upshape (dSD)} \qquad  (\forall u,v\in V)\ \ u \neq v \quad \Longrightarrow \quad (Eu\neq Ev \ \ \& \ \ uE \neq vE). $$
\end{enumerate}
\end{thm}

It is interesting to realise that finite semidistributive lattices are exactly those finite lattices whose dual digraphs have the ``separation property'' (S) strengthened to the ``strong separation property'' (dSD).

Now we switch to general (not necessarily finite) lattices. Our aim is to 
initiate a study of general join and meet semidistributive lattices by their dual digraphs. We will start with the following result and combine it with Theorem~\ref{SD:char} in the finite case.

\begin{prop}\label{lem:outneigh}
Let $\Lalg$ be a general lattice (with bounds) and consider the set $X_\Lalg$ of all MDFIPs of $\Lalg$. Let $x=\mdfip{F}{I}$ and $y=\mdfip{G}{J}$. Then we have 
$$
Ex=Ey \Longleftrightarrow I=J \quad \& \quad xE=yE \Longleftrightarrow F=G.
$$
\end{prop}

\begin{proof}
We prove the first equivalence and the second one can be shown by dual arguments. Assume $I=J$ and let $z=\mdfip{H}{K}$. Then
\[zEx \quad \Longleftrightarrow \quad H \cap I = \emptyset \quad \Longleftrightarrow \quad H \cap J = \emptyset
\quad \Longleftrightarrow \quad zEy. \]

Now let $Ex=Ey$ and suppose that $I \neq J$. Without loss of generality, let $I \nsubseteq J$. Then there exists $a \in I$ with $a \notin J$. Now consider the DFIP $\mdfip{{\uparrow}a}{J}$. Extend $J$ to $\hat{J}$, which is maximal with respect to being disjoint from ${\uparrow}a$. Now extend ${\uparrow}a$ to $\hat{G}$ and consider the MDFIP $z=\mdfip{\hat{G}}{\hat{J}}$. Since $\hat{G} \cap \hat{J} = \emptyset$ and $J \subseteq \hat{J}$, we have $\hat{G}\cap J = \emptyset$. So $z \in Ey$. Since $a \in \hat{G} \cap I$ we have $z \notin Ex$, a contradiction.
\end{proof}

From Proposition~\ref{lem:outneigh} and Theorem~\ref{SD:char} we immediately obtain the following result. We note that the letter ``{\upshape u}" stands for \emph{Urquhart} as we use his concept of MDFIPs to obtain the characterisations below.

\begin{coro}\label{cor:outneigh}
Let $\Lalg$ be a finite lattice (with bounds) and let $X_\Lalg$
be the set of all MDFIPs of $\Lalg$. Then
\begin{enumerate}[{\upshape (i)}]
\item $\Lalg$ is join-semidistributive if and only if it satisfies the condition
\[\text{\upshape (uJSD)} \qquad (\forall x=\mdfip{F}{I}, y=\mdfip{G}{J} \in X_\Lalg)\ \ x \neq y \quad \Longrightarrow \quad I \neq J.\]

\item $\Lalg$ is meet-semidistributive if and only if it satisfies the condition
\[\text{\upshape (uMSD)} \quad (\forall x=\mdfip{F}{I}, y=\mdfip{G}{J} \in X_\Lalg)\ \
x \neq y \quad \Longrightarrow \quad F \neq G.\]

\item $\Lalg$ is semidistributive if and only if it satisfies the condition
\[\text{\upshape (uSD)} \ \ (\forall x=\mdfip{F}{I}, y=\mdfip{G}{J} \in X_\Lalg)\ \  
x \neq y \ \ \Longrightarrow \ \  (I \neq J \  \& \  F \neq G).\]

\end{enumerate}
\end{coro}

We are able to show that the condition (dJSD) presented above, which characterises in the finite case the dual digraphs of join semidistributive lattices, works also in the general (infinite) case as a natural condition on the dual digraph which is sufficient 
for the lattice to be join semidistributive.

\begin{thm}\label{prop:SJimpliesJSD}
Let $\Lalg$ be a general lattice (with bounds).
If the dual digraph 
$\mathbf{X}_\Lalg=(X_\Lalg,E)$ of $\Lalg$ satisfies the 
condition (dJSD), then $\Lalg$ is join semidistributive.
\end{thm}
 
\begin{proof} Assume that the lattice $\Lalg$ is not join semidistributive. Then there exist $a,b,c \in L$ such that $a \vee b = a \vee c$ but $a\vee (b \wedge c) < a \vee b$. Consider the DFIP $\MDFIP{(a\vee b)}{(a \vee (b \wedge c))}$ and extend the ideal to $I$, which is maximal with respect to being disjoint from ${\uparrow}(a\vee b)$. (See also Figure~\ref{fig:sufficiency}.) 

We show that ${\uparrow}b$ is disjoint from $I$. If not, then since $a\vee (b \wedge c) \in I $ and $b \in I$, we get $(a \vee (b \wedge c))\vee b =a\vee b \in I$.  Similarly, if $c \in I$ then $a \vee c=a\vee b \in I$, and so ${\uparrow}c \cap I = \emptyset$.

Hence we can extend the filters ${\uparrow}b$  to $F$ and ${\uparrow} c$ to $G$ so  that both $F$ and $G$ are maximal with respect to being disjoint from $I$. We claim that $F \neq G$. We show that $b\notin G$. If $b \in G$, then $b \wedge c \in G$ and since $G$ is an up-set this would imply $a \vee (b \wedge c) \in G$, a contradiction.

Now with $x=\mdfip{F}{I}$ and $y=\mdfip{G}{I}$ we have $x \neq y$.  Using Proposition~\ref{lem:outneigh} we get $Ex=Ey$ and hence $\mathbf{X}_\Lalg$ does not satisfy (dJSD). 
\end{proof}

\begin{rem}
We note here why in the above proof of Theorem~\ref{prop:SJimpliesJSD}  we must extend the ideal ${\downarrow}a \vee (b \wedge c)$ and not just ${\downarrow}a$. Notice that in the example of the lattice in Figure~\ref{fig:sufficiency}, one extension of ${\downarrow}a$ is to ${\downarrow}d$, which would allow ${\uparrow}b$ and ${\uparrow}c$ to be extended to ${\uparrow}e$.

\begin{figure}[ht]
\centering
\begin{tikzpicture}[scale=0.8]
\begin{scope}
\node[unshaded] (top) at (0,4) {};
\node[unshaded] (b) at (-1,3) {};
\node[unshaded] (c) at (0,3) {};
\node[unshaded] (f) at (1,3) {};
\node[unshaded] (e) at (0,2) {};
\node[unshaded] (bot) at (0,1) {};
\node[unshaded] (a) at (1,2) {};
\node[unshaded] (d) at (2,3) {};
  \draw[order] (bot)--(e)--(b)--(top);
  \draw[order] (e)--(c)--(top);
\draw[order] (e)--(f)--(top);
\draw[order] (bot)--(a)--(f)--(top);
\draw[order] (a)--(d)--(top);
  \node[label,anchor=west,xshift=1pt] at (a) {$a$};
  \node[label,anchor=east,xshift=-1pt] at (c) {$c$};
  \node[label,anchor=west,xshift=1pt] at (d) {$d$};
  \node[label,anchor=east,xshift=-1pt] at (b) {$b$};
\node[label,anchor=east,xshift=-1pt] at (e) {$e$};

  \end{scope}
\end{tikzpicture}
\caption{Illustrating the proof of Theorem~\ref{prop:SJimpliesJSD}}\label{fig:sufficiency} 
\end{figure}
\end{rem}

By using dual arguments and then combining the two results, it follows that the conditions (dMSD) and (dSD) for the dual digraphs serve as sufficient conditions for the properties of lattice meet semidistributivity and semidistributivity, respectively. 

\begin{coro}\label{coro:(dSD)}
Let $\Lalg$ be a general lattice (with bounds). 
\begin{enumerate}[\normalfont (i)]
\item 
If the dual digraph $\mathbf{X}_\Lalg=(X_\Lalg,E)$ of $\Lalg$ satisfies the condition (dMSD), then $\Lalg$ is meet semidistributive.
\item If the dual digraph $\mathbf{X}_\Lalg=(X_\Lalg,E)$ of $\Lalg$ satisfies the condition (dSD), then $\Lalg$ is semidistributive.
\end{enumerate} 
\end{coro}

\begin{rem}\label{checking-u-versions}
Since we want to initiate the study of general join semidistributive lattices, meet semidistributive lattices  and semidistributive lattices by their dual spaces, we prefer to use above the conditions (dJSD), (dMSD) and (dSD), respectively, which are expressed via the vertices $u, v$  and via the ``downsets'' $Eu, Ev$ and the ``upsets'' $uE, vE$ in the dual digraphs $\mathbf{X}_\Lalg=(X_\Lalg,E)$ of lattices~$\Lalg$.

However, as we also show in the next examples, if we have the dual digraphs of the lattices presented explicitly as the filter-ideal pairs, then it might be easier to check the conditions (dJSD), (dMSD), and (dSD) on the dual digraphs by checking the conditions (uJSD), (uMSD), and (uSD) from Corollary~\ref{cor:outneigh}, these conditions are by Proposition~\ref{lem:outneigh} equivalent to the previous ones. The conditions (uJSD), (uMSD) and (uSD) can often be seen immediately on the dual digraphs as they only say that different vertices $x, y \in \mathbf{X}_\Lalg$ do not share the same filters or the same ideals.  
\end{rem}

We now illustrate our results above on an example of an infinite semidistributive lattice and on an example of an infinite meet semidistributive lattice that is not join semidistributive. We also illustrate our claim in the above remark that the conditions (uJSD), (uMSD) and (uSD) can often be seen immediately on the dual digraphs. 
   
\begin{figure}[ht]
\centering
\begin{tikzpicture}[scale=0.65]
\begin{scope}[yshift=2.5cm]
\node[unshaded] (bot) at (0,0) {};
\node[unshaded] (x1) at (-1,2) {};
\node[unshaded] (x0) at (-1,1) {};
\node[unshaded] (y1) at (1,2) {};
\node[unshaded] (y0) at (1,1) {};
\node[unshaded] (y) at (1,4) {};
\node[unshaded] (x) at (-1,4) {};
\node[unshaded] (top) at (0,5) {};
\node[] (rght) at (1,3) {$\vdots$};
\node[] (lft) at (-1,3) {$\vdots$};

\draw[order] (bot)--(x0)--(x1)--(lft)--(x)--(top)--(y)--(rght)--(y1)--(y0)--(bot);
\node[label,anchor=north] at (bot) {$0$};
\node[label,anchor=east,xshift=1pt] at (x0) {$a_0$};
\node[label,anchor=east,xshift=1pt] at (x1) {$a_1$};
\node[label,anchor=west,xshift=-2pt] at (y0){$b_0$};
\node[label,anchor=west,xshift=-2pt] at (y1){$b_1$};
\node[label,anchor=east,xshift=1pt] at (x) {$a_\omega$};
\node[label,anchor=west,xshift=-2pt] at (y){$b_\omega$};
\node[label,anchor=south] at (top) {$1$};
\end{scope}

\begin{scope}[xshift=8.7cm, scale=0.77]

\node[] (xI) at (-4, 2.5) {$a_\omega I_a$};
\node[] (lft) at (-4, 5) {$\vdots$};
\node[] (x3x2) at (-4, 7.5) {$a_3a_2$};
\node[] (x2x1) at (-4, 10) {$a_2a_1$};
\node[] (x1x0) at (-4, 12.5) {$a_1a_0$};

\node[] (yI) at (4, 12.5) {$b_\omega I_b$};
\node[] (rght) at (4, 10) {$\vdots$};
\node[] (y3y2) at (4, 7.5) {$b_3b_2$};
\node[] (y2y1) at (4, 5) {$b_2b_1$};
\node[] (y1y0) at (4, 2.5) {$b_1b_0$};

\node[] (x0y) at (0, 15) {$a_0b_\omega$};
\node[] (y0x) at (0, 0) {$b_0a_\omega$};

\path[thick,->,shorten >=3pt,shorten <=3pt] (y0x.north west) edge (xI.south east);
\path[thick,->,shorten >=3pt,shorten <=3pt] (y1y0.south west) edge (y0x.north east);

\path[thick,->,shorten >=3pt,shorten <=3pt] (x0y.south east) edge (yI.north west);
\path[thick,->,shorten >=3pt,shorten <=3pt] (x1x0.north east) edge (x0y.south west);

\path[thick,->,shorten >=3pt,shorten <=3pt] (xI.north) edge  (lft.south);
\path[thick,->,shorten >=3pt,shorten <=3pt] (lft.north) edge  (x3x2.south);
\path[thick,->,shorten >=3pt,shorten <=3pt] (x3x2.north) edge  (x2x1.south);
\path[thick,->,shorten >=3pt,shorten <=3pt] (x2x1.north) edge  (x1x0.south);

\path[thick,->,shorten >=3pt,shorten <=3pt] (yI.south) edge  (rght.north);
\path[thick,->,shorten >=3pt,shorten <=3pt] (rght.south) edge  (y3y2.north);
\path[thick,->,shorten >=3pt,shorten <=3pt] (y3y2.south) edge  (y2y1.north);
\path[thick,->,shorten >=3pt,shorten <=3pt] (y2y1.south) edge  (y1y0.north);

\path[thick,->,shorten >=5pt,shorten <=4pt] (xI.north west) [out=135,in=225] edge  (x1x0.south west);
\path[thick,->,shorten >=5pt,shorten <=4pt] (xI.north west) [out=135,in=225] edge  (x2x1.south west);
\path[thick,->,shorten >=5pt,shorten <=4pt] (xI.north west) [out=135,in=225] edge  (x3x2.south west);
\path[thick,->,shorten >=5pt,shorten <=4pt] (x3x2.north west) [out=130,in=230] edge  (x1x0.south west);

\path[thick,->,shorten >=5pt,shorten <=4pt] (yI.south east) [out=315,in=45] edge  (y1y0.north east);
\path[thick,->,shorten >=5pt,shorten <=4pt] (yI.south east) [out=315,in=45] edge  (y2y1.north east);
\path[thick,->,shorten >=5pt,shorten <=4pt] (yI.south east) [out=315,in=45] edge  (y3y2.north east);
\path[thick,->,shorten >=5pt,shorten <=4pt] (y3y2.south east) [out=310,in=40] edge  (y1y0.north east);

\path[thick,->,shorten >=5pt,shorten <=4pt] (x2x1.north east) edge  (x0y.south west);
\path[thick,->,shorten >=5pt,shorten <=4pt] (x3x2.north east) edge  (x0y.south west);
\path[thick,->,shorten >=5pt,shorten <=4pt] (xI.north east) edge  (x0y.south west);

\path[thick,->,shorten >=5pt,shorten <=4pt] (y2y1.south west) edge  (y0x.north east);
\path[thick,->,shorten >=5pt,shorten <=4pt] (y3y2.south west) edge  (y0x.north east);
\path[thick,->,shorten >=5pt,shorten <=4pt] (yI.south west) edge  (y0x.north east);

\path[thick,->,shorten >=5pt,shorten <=4pt] (y0x.north west) edge  (x3x2.south east);
\path[thick,->,shorten >=5pt,shorten <=4pt] (y0x.north west) edge  (x2x1.south east);
\path[thick,->,shorten >=5pt,shorten <=4pt] (y0x.north west) edge  (x1x0.south east);

\path[thick,->,shorten >=5pt,shorten <=4pt] (x0y.south east) edge  (y3y2.north west);
\path[thick,->,shorten >=5pt,shorten <=4pt] (x0y.south east) edge  (y2y1.north west);
\path[thick,->,shorten >=5pt,shorten <=4pt] (x0y.south east) edge  (y1y0.north west);

\end{scope}

\end{tikzpicture}
\caption{The infinite semidistributive lattice
$\mathbf{O}_\omega$,  
and the core of its dual space} 
\label{fig:dual-examples2}
\end{figure}

To simplify notation, we will  write $FI$ for $\mdfip{F}{I}$; in case $F = {\uparrow}a$ or $I = {\downarrow}b$, we  simply write $aI$ or $Fb$ 
or $ab$. 

\begin{eg}\label{Example:SD} 
(i) [\textbf{An infinite semidistributive lattice}]
Let $\mathbf{O}_\omega$ 
be the lattice with infinite chains $0<a_0<a_1<\dots<a_\omega<1$ and $0<b_0<b_1<\dots<b_\omega<1$ (see Figure~\ref{fig:dual-examples2}).
The dual space $\mathcal{D}(\mathbf{O}_\omega) = (X_{\mathbf{O}_\omega}, E,\T_{\mathbf{O}_\omega})$ has as its underlying set the set of MDFIPs 
\[X_{\mathbf{O}_\omega} = \{a_0b_\omega, a_{1}a_0, a_{2}a_1, ..., a_\omega I_a\}\, \cup \,\{b_0a_\omega, b_{1}b_0, b_{2}b_1, ..., b_\omega I_b\},\]
where $I_a = \{0, a_0, a_1, a_2,...\}$ and $I_b = \{0, b_0, b_1, b_2,...\}$. One can check that the MDFIPs that naturally arise in the edge relation $E$ are given by 
\begin{multline*}
        a_{j+1}a_jEa_{i+1}a_i, \quad \quad a_{j+1}a_j E a_{0}b_\omega, \quad \quad 
        b_0a_\omega E a_\omega I_a, 
        \quad \quad 
        b_0a_\omega E a_{i+1} a_i, \\
        a_{i+1}a_iEb_{k+1}b_k, \quad \quad a_{i+1}a_iEb_{\omega}I_b, \quad \quad a_{\omega}I_aEb_{k+1}b_k, \quad \quad a_{\omega}I_aEb_{\omega}I_b 
\end{multline*}
for all $i,j,k \in \omega$ with $i<j$. Swapping all $a$'s with $b$'s above will give us, with the previous sublist, the list of all MDFIPs in the edge relation $E$. The lattice and it dual space are drawn  in Figure~\ref{fig:dual-examples2}. For improved readability the double arrows $a_{i+1}a_iEb_{k+1}b_k$ and $b_{k+1}b_kEa_{i+1}a_i$ are not presented. 
 
To describe the basic open sets of the topology, we note that it is not hard to compute the intersections of subbasic open sets in $\mathcal{D}(\mathbf{O}_\omega)$. 
In particular for all $i,k\in \omega$ we have $V_{a_i}\cap W_{b_k} = \{b_0a_\omega\} = W_{b_0} = V_{a_{\omega}}$ and $V_{b_k}\cap W_{a_i} = \{a_0b_\omega\} = W_{a_0} = V_{b_\omega}$. If $a_i < a_k$ we have $V_{a_k}\cap W_{a_i}=\varnothing$ and $V_{a_i}\cap W_{a_k} = \{a_{i+1}a_i, a_{i+2}a_{i+1}, ..., a_ka_{k-1}\}$. Similar intersections hold for the pairs of $b_i$ and $b_k$. We also have that $V_{a_i}\cap W_{a_\omega} = \{a_{i+1}a_i, a_{i+2}a_{i+1},\dots, a_\omega I_a\}$, which is infinite. In fact, the only infinite basic open sets are of the form $V_{a_i}\cap W_{a_\omega}$ and $V_{b_i}\cap W_{b_{\omega}}$. 

It is easy to see that in the dual digraph any two different vertices do not share the same ideal and do not share the same filter. Hence the dual digraph satisfies the condition (uSD), which by Corollary~\ref{coro:(dSD)}(ii) and Remark~\ref{checking-u-versions} witnesses that  $\mathbf{O}_\omega$ is semidistributive.

(ii) [\textbf{An infinite meet semidistributive 
lattice, which is not join semidistributive}]

\begin{figure}[ht]
\centering
\begin{tikzpicture}[scale=0.7]
\begin{scope}[xshift = -2cm, yshift = -2cm]
\node[unshaded] (bot) at (0,0) {};
\node[unshaded] (x1) at (-1,2) {};
\node[unshaded] (x0) at (-1,1) {};
\node[unshaded] (y1) at (1,2) {};
\node[unshaded] (y0) at (1,1) {};
\node[unshaded] (yi) at (1,4) {};
\node[unshaded] (xi) at (-1,4) {};
\node[unshaded] (x) at (-1,5) {};
\node[unshaded] (y) at (1,5) {};
\node[unshaded] (z) at (0,5) {};
\node[unshaded] (top) at (0,6) {};
\node[] (rght) at (1,3) {$\vdots$};
\node[] (lft) at (-1,3) {$\vdots$};

\draw[order] (bot)--(x0)--(x1)--(lft)--(xi)--(z)--(yi)--(rght)--(y1)--(y0)--(bot);
\draw[order] (xi)--(x)--(top)--(y)--(yi);
\draw[order] (z)--(top);
\node[label,anchor=north] at (bot) {$0$};
\node[label,anchor=east,xshift=1pt] at (x0) {$a_0$};
\node[label,anchor=east,xshift=1pt] at (x1) {$a_1$};
\node[label,anchor=west,xshift=-2pt] at (y0){$b_0$};
\node[label,anchor=west,xshift=-2pt] at (y1){$b_1$};
\node[label,anchor=east,xshift=1pt] at (xi) {$a_\omega$};
\node[label,anchor=west,xshift=-2pt] at (yi){$b_\omega$};
\node[label,anchor=east,xshift=1pt] at (x) {$a$};
\node[label,anchor=west,xshift=-2pt] at (y){$b$};
\node[label,anchor=east] at (z) {$c$};
\node[label,anchor=south] at (top) {$1$};
\end{scope}

\begin{scope}[xshift=5cm, scale=0.85]

\node[] (xI) at (0, 3) {$a_\omega I_a$}; 
\node[] (rght) at (0, 5) {$\vdots$};
\node[] (x1x0) at (0, 7) {$a_1a_0$};
\node[] (x0y) at (0, 9) {$a_0b$};

\node[] (yI) at (0, -3) {$b_\omega I_b$}; 
\node[] (lft) at (0, -5) {$\vdots$};
\node[] (y1y0) at (0, -7) {$b_1b_0$};
\node[] (y0x) at (0, -9) {$b_0a$};

\node[] (yz) at (0, -1) {$bc$};
\node[] (xz) at (0, 1) {$ac$};

\path[thick,->,shorten >=3pt,shorten <=3pt] (rght.north) edge  (x1x0.south);
\path[thick,->,shorten >=3pt,shorten <=3pt] (lft.south) edge  (y1y0.north);

\path[thick,->,shorten >=3pt,shorten <=3pt] (x1x0.north) edge  (x0y.south);
\path[thick,->,shorten >=3pt,shorten <=3pt] (xI.north) edge  (rght.south);
\path[thick,->,shorten >=3pt,shorten <=3pt] (xz.north) edge  (xI.south);

\path[thick,->,shorten >=3pt,shorten <=3pt] (y1y0.south) edge  (y0x.north);
\path[thick,->,shorten >=3pt,shorten <=3pt] (yI.south) edge  (lft.north);
\path[thick,->,shorten >=3pt,shorten <=3pt] (yz.south) edge  (yI.north);

\path[thick,shorten <=2.5pt, shorten >=2.5pt,<->] (yz.north) edge  (xz.south);

\path[thick,->,shorten >=5pt,shorten <=4pt] (xz.north west) [out=130,in=225] edge  (x1x0.south west);
\path[thick,->,shorten >=5pt,shorten <=4pt] (xz.north west) [out=130,in=225] edge  (x0y.south west);

\path[thick,->,shorten >=5pt,shorten <=4pt] (yz.north east) [out=45,in=-45] edge  (x1x0.south east);
\path[thick,->,shorten >=4pt,shorten <=4pt] (yz.north east) [out=45,in=-10] edge  (x0y.south east);

\path[thick,->,shorten >=5pt,shorten <=4pt] (xI.north east) [out=45,in=-45] edge  (x0y.south east);
\path[thick,->,shorten >=4pt,shorten <=4pt] (xI.north west) [out=110,in=-100] edge  (x1x0.south west);

\path[thick,->,shorten >=5pt,shorten <=4pt] (yz.south east) [out=-45,in=45] edge  (y1y0.north east);
\path[thick,->,shorten >=5pt,shorten <=4pt] (yz.south east) [out=-45,in=45] edge  (y0x.north east);

\path[thick,->,shorten >=5pt,shorten <=4pt] (yI.south east) [out=-45,in=75] edge  (y0x.north east);
\path[thick,->,shorten >=5pt,shorten <=4pt] (yI.south west) [out=-100,in=95] edge  (y1y0.north west);

\path[thick,->,shorten >=5pt,shorten <=4pt] (xz.south west) [out=225,in=130] edge  (y1y0.north west);
\path[thick,->,shorten >=5pt,shorten <=4pt] (xz.south west) [out=225,in=130] edge  (y0x.north west);
\end{scope}

\end{tikzpicture}
\caption{The infinite meet but not join semidistributive
lattice
$\hat{\mathbf{O}}_\omega$,   
and the core of its dual space} 
\label{fig:dual-ex2} 
\end{figure}

Let $\hat{\mathbf{O}}_\omega$ ($\mathbf{O}_\omega$ with a ``hat'') 
    be the lattice with two infinite chains $0<a_0<a_1<\dots<a_\omega<a<1$ and $0<b_0<b_1<\dots<b_\omega < b<1$ and an element $c$ such that $a,b,c$ are incomparable and $c = a_\omega \vee b_\omega$. 
    (See Figure~\ref{fig:dual-ex2}.) 
    
One can check that MDFIPs, in which $a$'s appear, naturally arising in the edge relation $E$ are 
given by 
\begin{gather*}
a_{j+1}a_jEa_{i+1}a_i, \quad \quad 
    a_{i+1}a_iEb_{k+1}b_k, \\
a_{j+1}a_j E a_{0}y, \quad \quad 
b_0a E a_\omega I_a,  
\quad \quad b_0a E a_{i+1} a_i, \quad \quad
a_{i+1} a_i E b_0a, \\ a_{i+1}a_iEb_{\omega}I_b, \quad \quad a_{\omega}I_aEb_{k+1}b_k, \quad \quad       a_{\omega}I_aEb_{\omega}I_b, \\ 
ac E bc, \quad \quad ac E a_{j+1}a_j, \quad \quad ac E a_\omega I_a 
\end{gather*}
for all $i,j,k \in \omega$ with $i<j$. Swapping all $a$'s with $b$'s above will give us, with the previous sublist, the list of all MDFIPs in the edge relation $E$.  
The dual space is drawn 
    in~Figure~\ref{fig:dual-ex2} 
    with the double arrows of the forms $a_{i+1}a_iEb_{k+1}b_k$ and $b_{k+1}b_kEa_{i+1}a_i$ 
    not being presented 
   to make the diagram more readable. 

   The dual space
    $\mathcal{D}(\hat{\mathbf{O}}_\omega)$
    has base set 
    \[
    \{a_0b, a_{1}a_0, a_{2}a_1, ..., a_\omega I_a, ac\} \cup \{b_0a, b_{1}b_0, b_{2}b_1, ..., b_\omega I_b, bc\}
    \]
    where $I_a = \{0, a_0, a_1, a_2,...\}$ and $I_b = \{0, b_0, b_1, b_2,...\}$. 

To describe the basic open sets of the topology, we note that the intersections of the subbasic open sets are very similar to the previous example. The difference is that 
in this case there are four sets, which can create more infinite basic open sets, namely 
the intersections $W_a\cap V_{a_i}$, $W_b\cap V_{b_i}$, $W_c\cap V_{a_i}$ and $W_c\cap V_{b_i}$ are infinite for all $i \in \omega$. 

One can immediately see that in the dual digraph any two different vertices do not share the same filter. Hence the dual digraph satisfies the condition (uMSD), which by Corollary~\ref{coro:(dSD)}(ii) and Remark~\ref{checking-u-versions} witnesses that the lattice $\hat{\mathbf{O}}_\omega$ is meet semidistributive.

On the other hand, the dual digraph of the lattice $\hat{\mathbf{O}}_\omega$  does not satisfy the condition 
 (uJSD), and thus the condition (dJSD), 
    since the 
    vertices 
    $ac$ and $bc$ share the same ideal.  This lattice is not 
   join semidistributive as on the lattice side we 
   easily  
   see that  $c\vee a= c\vee b  =1$ but $c\vee(a\wedge b)=c$. 
\end{eg}

The two parts of Example~\ref{Example:SD} as well as the part (i) of the following example might lead the reader to think that our conditions  (dJSD), (dMSD) and (dJSD) might be also necessary conditions on the dual digraphs 
for the lattices to be join semidistributive, meet semidistributive and semidistributive, respectively. Yet the part (ii) of the following example shows that this is not the case. 

\begin{figure}[ht]
\centering
\begin{tikzpicture}[scale=0.7]
\begin{scope}[]
\node[unshaded] (b0) at (0,0) {};
\node[unshaded] (b1) at (0,1.5) {};
\node[unshaded] (b2) at (0,3) {};
\node[unshaded] (b3) at (0,4.5) {};

\node[unshaded] (a1) at (-2,3) {};
\node[unshaded] (a3) at (-1.25,6) {};

\node[unshaded] (c0) at (2,1.5) {};
\node[unshaded] (c2) at (1.25,4.5) {};

\node[unshaded] (b) at (0,7.5) {};

\node[] (rght) at (1,7.5) {$\vdots$};
\node[] (lft) at (-1,7.5) {$\vdots$};
\node[] (mdle) at (0,5.75) {$\vdots$};
\node[unshaded] (top) at (0,9) {};

\draw[order] (b0)--(b1)--(b2)--(b3)--(mdle)--(b)--(top);
\draw[order] (b1)--(a1)--(a3)--(lft)--(top);
\draw[order] (b0)--(c0)--(c2)--(rght)--(top);
\draw[order] (b2)--(c2);
\draw[order] (b3)--(a3);

\node[label,anchor=east] at (b0) {$b_0$};
\node[label,anchor=west] at (b1) {$b_1$};
\node[label,anchor=east] at (b2) {$b_2$};
\node[label,anchor=west] at (b3) {$b_3$};
\node[label,anchor=west, yshift=-5pt] at (b) {$b_\omega$};
\node[label,anchor=south] at (top) {$1$};
\node[label,anchor=east] at (a1) {$a_1$};
\node[label,anchor=east] at (a3) {$a_3$};
\node[label,anchor=west] at (c0) {$c_0$};
\node[label,anchor=west] at (c2) {$c_2$};

\end{scope}

\begin{scope}[xshift=8.2cm, scale=0.77]

\node[] (bIab) at (-3, 16) {$b_{\omega} I_{ab}$};
\node[] (bIbc) at (3, 16) {$b_\omega I_{bc}$};

\node[] (c0Iab) at (-4, -1) {$c_0I_{ab}$};
\node[] (a1Ibc) at (4, -1) {$a_1I_{bc}$};

\node[] (c0b) at (-2, 2) {$c_0b_\omega$};
\node[] (a1b) at (2, 2) {$a_1b_\omega$};

\node[] (mddl) at (0, 14) {$\vdots$};
\node[] (b4a3) at (0, 11.5) {$b_4a_3$};
\node[] (b3c2) at (0, 9) {$b_3c_2$};
\node[] (b2a1) at (0, 6.5) {$b_2a_1$};
\node[] (b1c0) at (0, 4) {$b_1c_0$};

\path[thick,->,shorten >=3pt,shorten <=3pt] (mddl.south) edge  (b4a3.north);
\path[thick,->,shorten >=3pt,shorten <=3pt] (b4a3.south) edge  (b3c2.north);
\path[thick,->,shorten >=3pt,shorten <=3pt] (b3c2.south) edge  (b2a1.north);
\path[thick,->,shorten >=3pt,shorten <=3pt] (b2a1.south) edge  (b1c0.north);

\path[thick,shorten <=2.5pt, shorten >=2.5pt,<->] (bIbc.west) edge  (bIab.east);
\path[thick,shorten <=2.5pt, shorten >=2.5pt,<->] (bIbc)  edge  (bIab);

\path[thick,->,shorten >=3pt,shorten <=3pt] (bIab)  edge  (mddl);
\path[thick,->,shorten >=3pt,shorten <=3pt] (bIab)  edge  (b4a3);
\path[thick,->,shorten >=3pt,shorten <=3pt] (bIab)  edge  (b3c2);
\path[thick,->,shorten >=3pt,shorten <=3pt] (bIab)  edge  (b2a1);
\path[thick,->,shorten >=3pt,shorten <=3pt] (bIab)  edge  (b1c0);

\path[thick,->,shorten >=3pt,shorten <=3pt] (bIbc)  edge  (mddl);
\path[thick,->,shorten >=3pt,shorten <=3pt] (bIbc)  edge  (b4a3);
\path[thick,->,shorten >=3pt,shorten <=3pt] (bIbc)  edge  (b3c2);
\path[thick,->,shorten >=3pt,shorten <=3pt] (bIbc)  edge  (b2a1);
\path[thick,->,shorten >=3pt,shorten <=3pt] (bIbc)  edge  (b1c0);

\path[thick,shorten <=2.5pt, shorten >=2.5pt,<->] (c0b.east) edge  (a1b.west);
\path[thick,shorten <=2.5pt, shorten >=2.5pt,<->] (c0b)  edge  (c0Iab);
\path[thick,shorten <=2.5pt, shorten >=2.5pt,<->] (a1b)  edge  (a1Ibc);
\path[thick,->,shorten >=3pt,shorten <=3pt] (c0Iab)  edge  (a1b);
\path[thick,->,shorten >=3pt,shorten <=3pt] (a1Ibc)  edge  (c0b);

\path[thick,->,shorten >=3pt,shorten <=3pt] (c0b) [out=67.5,in=225] edge  (mddl);
\path[thick,->,shorten >=3pt,shorten <=3pt] (c0b) [out=67.5,in=240] edge  (b4a3);
\path[thick,->,shorten >=1pt,shorten <=3pt] (c0b) [out=67.5,in=240] edge  (b2a1);

\path[thick,->,shorten >=3pt,shorten <=3pt] (a1b) [out=112.5,in=315] edge  (mddl);
\path[thick,->,shorten >=0.1pt,shorten <=3pt] (a1b) [out=112.5,in=305] edge  (b3c2);
\path[thick,->,shorten >=3pt,shorten <=3pt] (a1b)  edge  (b1c0);

\path[thick,shorten <=2.5pt, shorten >=2.5pt,<->] (c0Iab) [out=95,in=225] edge (bIab);
\path[thick,shorten <=2.5pt, shorten >=2.5pt,<->] (a1Ibc) [out=85,in=315] edge (bIbc);

\path[thick,->,shorten >=3pt,shorten <=3pt] (c0b) [out=95,in=260] edge  (bIab);
\path[thick,->,shorten >=3pt,shorten <=3pt] (a1b) [out=85,in=280] edge  (bIbc);

\end{scope}

\end{tikzpicture}
\caption{The infinite lattice $\mathbf{R}$, which is neither join semidistributive nor meet semidistributive, and its dual space.}\label{fig:dual-ex3}
\end{figure}

\begin{eg}
(i) [\textbf{The ``rocket'', an infinite lattice that is neither join semidistributive nor meet semidistributive}]\label{Rocket}

Let $\mathbf{R}$ be the lattice on left in~Figure~\ref{fig:dual-ex3} (cf.~\cite[Fig. 3]{DPR-75}) with three infinite chains: $b_0< b_1< b_2< ...< b_\omega<1$, then $b_1<a_1< a_3<a_5< ...< 1$ and finally $b_0<c_0<c_2<c_4<...<1$. Moreover $b_i <a_i$ for all odd $i \in \omega$ and $b_k < c_k$ for all even $k\in \omega$. The dual space of the lattice $\mathbf{R}$ is drawn in~Figure~\ref{fig:dual-ex3}, where $I_{ab}=\{a_{2i+1}\}_{i\in \omega}\cup \{b_j\}_{j \in \omega}$ and $I_{bc}=\{c_{2i}\}_{i \in \omega }\cup \{b_j\}_{j \in \omega }$. To ``reduce clutter'' and present the dual space more transparently, some relations from the pairs $c_0I_{ab}$ and $a_1I_{bc}$ have not been drawn in. This lattice can easily be seen not to be join semidistributive, since on the lattice side we have $b_\omega \vee a_1 = b_\omega \vee c_0 = 1$, but $b_\omega \vee (a_1\wedge c_0) = b_\omega$. Nor is this lattice meet semidistributive, since  $a_1\wedge b_{\omega} = a_1\wedge c_2 = b_1$ whereas $a_1\wedge (b_{\omega}\vee c_2) = a_1$. On the digraph we can see the pairs $c_0b_\omega$ and $a_1b_\omega$ share the same ideal, thus the dual does not satisfy the condition (uJSD), and thus the condition (dJSD). Also the pairs $b_\omega I_{ab}$ and $b_\omega I_{bc}$ demonstrate the failure of (uMSD) and hence (dMSD). 

To complete the description of the dual space and to describe the basic open sets 
of the topology, we note that the subbasic sets fall into four types. The first 
type of subbasic sets are those 
pertaining to the $a_{2i+1}$:  
for $i\in \omega$ 
we have that 
$W_{a_{2i+1}} = \{b_{2j+1}c_{2j}\}_{j \leq i}\cup 
\{b_{2(j+1)}a_{2j-1}\}_{j\leq i} \cup 
\{a_1b_\omega, a_1I_{bc}\}$ 
and $V_{a_{2i+1}} = 
\{b_{2j+2}a_{2j+1}\}_{i\leq j}
\cup \{c_0I_{ab}, b_\omega I_{ab}\}$. The second 
type 
pertains to the $c_{2i}$'s:  
for $i\in \omega$  
we get $W_{c_{2i}} =
\{b_{2j+1}c_{2j}\}_{1\leq j \leq i} \cup 
\{b_{2j}a_{2j+1} \}_{1 \leq j \leq i} 
\cup 
\{c_0b_\omega, c_0I_{ab}\}$ and $V_{c_{2i}} = 
\{b_{2j+1}c_{2j}\}_{i\leq j}
\cup \{a_1I_{bc}, b_\omega I_{bc}\}$. For the final two 
types, 
set $d_{k} = c_k$ if $k$ is even and $d_{k} = a_k$ if $k$ is odd 
for $k\in \omega$. Then the subbasic sets associated to $b_{\omega}$ can be described as $W_{b_\omega}=\{b_{j+1}d_j \}_{j\in \omega}\cup\{b_\omega I_{ab}, b_\omega I_{bc}\}$ and $V_{b_\omega}=\{c_0b_\omega, a_1b_\omega\}$. The final type are subbasic sets of the form $V_{b_i}=\{c_0I_{ab}, c_0b_\omega, a_1I_{bc}, a_1b_\omega, b_\omega I_{ab}, b_\omega I_{bc}\} \cup \{b_{2j+1}c_{2j}\}_{2j \geq i} \cup \{b_{2j+2}a_{2j+1} \}_{ 2j+1\geq  i}$ and $W_{b_i} = \{b_{j}d_{j-1}\}_{j\leq i}$.

\begin{figure}[ht]
\centering
\begin{tikzpicture}[scale=0.8]
\begin{scope}[]
\node[unshaded] (b0) at (0,0) {};
\node[unshaded] (b1) at (0,1.5) {};
\node[unshaded] (b2) at (0,3) {};
\node[unshaded] (b3) at (0,4.5) {};

\node[unshaded] (a1) at (-2,3) {};
\node[unshaded] (a3) at (-1.25,6) {};

\node[unshaded] (c0) at (2,1.5) {};
\node[unshaded] (c2) at (1.25,4.5) {};

\node[unshaded] (a) at (-0.75,8) {};
\node[unshaded] (b) at (0,7.25) {};
\node[unshaded] (c) at (0.75,8) {};
\node[unshaded] (top) at (0,8.75) {};

\node[] at (-1.15,6.4) {$\cdot$};
\node[] at (-1.05,6.8) {$\cdot$};
\node[] at (-0.95,7.2) {$\cdot$};
\node[] at (-0.85,7.6) {$\cdot$};

\node[] at (1.18,5) {$\cdot$};
\node[] at (1.11,5.5) {$\cdot$};
\node[] at (1.04,6) {$\cdot$};
\node[] at (0.97,6.5) {$\cdot$};
\node[] at (0.9,7) {$\cdot$};
\node[] at (0.83,7.5) {$\cdot$};

\node[] at (0,4.95) {$\cdot$};
\node[] at (0,5.4) {$\cdot$};
\node[] at (0,5.85) {$\cdot$};
\node[] at (0,6.3) {$\cdot$};
\node[] at (0,6.75) {$\cdot$};

\draw[order] (b0)--(b1)--(b2)--(b3);
\draw[order] (b1)--(a1)--(a3);
\draw[order] (b0)--(c0)--(c2);
\draw[order] (b2)--(c2);
\draw[order] (b3)--(a3);
\draw[order] (b)--(a)--(top)--(c)--(b);

\node[label,anchor=east] at (b0) {$b_0$};
\node[label,anchor=west] at (b1) {$b_1$};
\node[label,anchor=east] at (b2) {$b_2$};
\node[label,anchor=west] at (b3) {$b_3$};
\node[label,anchor=west, yshift=-3pt, xshift=-2pt] at (b) {$b_\omega$};
\node[label,anchor=south] at (top) {$1$};
\node[label,anchor=east] at (a1) {$a_1$};
\node[label,anchor=east] at (a3) {$a_3$};
\node[label,anchor=east] at (a) {$a_\omega$};
\node[label,anchor=west] at (c0) {$c_0$};
\node[label,anchor=west] at (c2) {$c_2$};
\node[label,anchor=west] at (c) {$c_\omega$};

\end{scope}

\end{tikzpicture}
\caption{An infinite lattice that satisfies JSD and MSD, but does not satisfy (uMSD) and hence its dual  does not satisfy (dMSD).}\label{fig:ex4}
\end{figure}

[(ii) \textbf{An infinite semidistributive lattice, whose dual fails (dMSD)}]

As in Example~\ref{Example:SD} where we modified (i) to obtain (ii),  here we modify $\mathbf{R}$ from part (i). Let $\bar{\mathbf{R}}$ be the lattice $\mathbf{R}$ with the elements $a_\omega$ and $c_\omega$ added. (This is in fact dual to the lattice in Figure 10 of \cite{JR79}.)  The lattice $\bar{\mathbf{R}}$ is easily seen (on the ``lattice side'') to satisfy both JSD and MSD, hence it is semidistributive. Yet the dual digraph of $\bar{\mathbf{R}}$  does not satisfy the condition (uMSD) since the filter ${\uparrow}b_\omega$ can be maximally paired with either the ideal $\{a_{2i+1}\}_{i\in\omega} \cup \{b_j\}_{j\in \omega}$ or the ideal $\{c_{2i}\}_{i\in\omega} \cup \{b_j\}_{j\in\omega}$.  
\end{eg}

From the above it follows that to characterise the duals of general (infinite) lattices that are join semidistributive or meet semidistributive or semidistributive, one needs to weaken our conditions (dJSD), (dMSD) and (dSD), respectively. 

\begin{prob}\label{prob1}
\emph{What are necessary and sufficient conditions for the dual of a general (infinite) lattice $\Lalg$ so that the lattice  $\Lalg$ is join semidistributive (resp. meet semidistributive resp. semidistributive)?}
\end{prob}

The conditions (dJSD), (dMSD) and (dSD) in Theorem~\ref{prop:SJimpliesJSD} and Corollary~\ref{coro:(dSD)}, or their better visualisable versions (uJSD), (uMSD) and (uSD) in accordance to Remark~\ref{checking-u-versions}, strengthen the (S) condition of the definition of TiRS digraphs. We formulate now our second open problem:

\begin{prob}\label{prob2}
\emph{Can the conditions (dJSD), (dMSD) and (dSD) as  strengthened versions of the (S) condition for the duals of general lattices interact with the topological conditions of Plo\v{s}\v{c}ica spaces (Definition~\ref{def:ploscicaspaces}) to produce simplified definitions of the dual spaces of general join-semidistributive lattices, meet-semidistributive lattice and  semidistributive lattices, respectively?}
\end{prob}

Some of the most prominent semidistributive lattices are free lattices. For three or more generators, it is well-known that the free lattices are infinite. We hope that the results of Sections~\ref{sec:ploscicaspaces} and~\ref{sec:dualspacesSDL} here can be used as a platform to study free lattices via their dual spaces. The dual digraphs and spaces would of course have to be modified to accommodate the lack of bounds of the lattices. The dual digraphs would then have both a sink and a source, in a similar manner to the bounded Priestley spaces dual to unbounded distributive lattices (cf.~\cite[Section~1.2 and Theorem~4.3.2]{CD98}). 

We are ready to present our third and final open problem: 

\begin{prob}\label{prob3}
\emph{Identify conditions on the dual TiRS digraph that will correspond to Whitman's condition~\cite{Whit41} on the lattice. More generally, describe the dual spaces of the free lattices.}
\end{prob}

\subsection*{Acknowledgements}

The first author acknowledges the hospitality of Matej Bel University during a visit in August-September 2024, 
as well as the National Research Foundation (NRF) of South Africa grant 127266 and the Slovak National Scholarship Programme. 
The second author acknowledges support by Slovak VEGA grant 1/0152/22 and his 
his appointment as a Visiting Professor at the University
of Johannesburg from June 2020. 
The third author is grateful to be funded by the IRMIA++ Interdisciplinary Thematic Institute of the ITI 2021-2028 programme of the University of Strasbourg, the CNRS and Inserm. It has received financial support from the IdEx Unistra (ANR-10-IDEX-0002), and funding under the Programme d'Investissements d'Avenir as part of the SFRI-STRAT'US project (ANR-20-SFRI-0012). 

The authors are grateful to the referees of the original submission of this paper for their valuable 
comments. In particular, they pointed us to~\cite[Section 11]{BCM25} and made many important suggestions that led to an improvement of Section~\ref{sec:dualspacesSDL}.

The first two authors would like to express their gratitude to Hilary Priestley, to whom this paper is dedicated,
for enabling them 
to be firstly her students and later her collaborators.  
This supervision and the collaboration with her  have been crucial for their further mathematical careers.

\end{document}